\documentclass[10pt,letterpaper]{amsart}
\usepackage[utf8]{inputenc}
\usepackage[english]{babel}
\usepackage{subfigure}
\usepackage{amsmath}
\usepackage{amsthm} 
\usepackage{hyperref} 
\usepackage{amsfonts}
\usepackage{tikz-cd}
\usepackage{amssymb}
\usepackage{xcolor}
\usepackage{csquotes}
\usepackage[
backend=biber,
style=alphabetic,
sorting=nty,doi=false,isbn=false, url=false
]{biblatex}

\theoremstyle{plain}
\newtheorem*{theorem*}{Theorem}
\newtheorem{theorem}{Theorem}[section]
\newtheorem{theo}{Theorem}
\theoremstyle{definition}
\newtheorem{prop}[theorem]{Proposition}
\newtheorem{defi}[theorem]{Definition}
\newtheorem{lemma}[theorem]{Lemma}
\newtheorem{example}[theorem]{Example}

\newtheorem*{corollary*}{Corollary}
\newtheorem*{prop*}{Proposition}
\newtheorem*{lemma*}{Lemma}

\newcommand{\C}{\mathbb{C}}
\newcommand{\Z}{\mathbb{Z}}
\newcommand{\Pro}{\mathbb{P}}
\newcommand{\Q}{\mathbb{Q}}
\newcommand{\N}{\mathbb{N}}
\newcommand{\R}{\mathbb{R}}

\newcommand{\eff}{\text{eff}}
\newcommand{\et}{\text{\'et}}

\newcommand{\spc}{\text{Spec}}

\newcommand{\ho}{\text{Hom}}

\newcommand{\CH}{\text{CH}}

\newtheoremstyle{named}{}{}{\itshape}{}{\bfseries}{.}{.5em}{\thmnote{#3's }#1}
\theoremstyle{named}

\theoremstyle{remark}
\newtheorem{remark}[theorem]{Remark}
\tikzcdset{
  cells={font=\everymath\expandafter{\the\everymath\displaystyle}},
}

\newcommand{\info}{{
  \bigskip
  \footnotesize

  \textsc{Institut de Math\'ematiques de Bourgogne, UMR 5584 CNRS, Universit\'e Bourgogne Franche-Comt\'e, F-21000 Dijon, France}\par\nopagebreak
  \textit{E-mail address}: \texttt{ivan-alejandro.rosas-soto@u-bourgogne.fr}
  }}


\usepackage[left=3cm,right=3.5cm,top=2cm,bottom=3.2cm]{geometry}
\author{Iv\'an Rosas-Soto}
\date{17th July 2023}
\title{\'Etale degree map and 0-cycles}
\begin{document}
\maketitle
\begin{abstract}
Using the triangulated category of \'etale motives over a field $k$, for a smooth projective variety $X$ over $k$, we define the group $\text{CH}^\text{\'et}_0(X)$ as an \'etale analogue of 0-cycles. We study the properties of $\text{CH}^\text{\'et}_0(X)$ and give a description of the birational invariance of such a group. We define and present the \'etale degree map using Gysin morphisms in \'etale motivic cohomology and the \'etale index as an analogue to the classical case. We give examples of smooth projective varieties over a field $k$ without zero cycles of degree one but with \'etale zero cycles of degree one, but this property is not always true as we give examples where the \'etale degree map is not surjective.
\end{abstract}
\tableofcontents

\section{Introduction}

Let $X$ be a smooth projective variety over a field $k$. We define the zero cycles of $X$, denoted by $Z_0(X)$, as the free abelian group generated by sums $\sum_x n_x x$ where $x$ is a closed point of $X$ and $n_x \in \Z$ is zero for all but finitely many $x$. The degree map is defined by
\begin{align*}
    \deg:Z_0(X) &\to \Z \\
    \sum_x n_x x &\mapsto \sum_x n_x [k(x):k].
\end{align*}

This map is compatible with the quotient by rational equivalence, so we can define it over $\CH_0(X)$. By definition, it coincides with the push-forward along the structural map $g:X\to \spc(k)$ as $g_*:\CH_0(X)\to \CH_0(\spc(k))=\Z$. We define the index of a variety $X$ over $k$ as follows
\begin{align*}
    I(X):= \gcd\left\{[k(x):k] \ | \ x \in X\right\}.
\end{align*} 

If the field is algebraically closed, then there exists a $k$-rational point and the degree map is surjective. However, if the field is not algebraically closed, the existence of a $k$-rational point, or even of a zero cycle of degree 1, is not guaranteed. Note that the existence of a $k$ rational point implies the existence of a zero cycle of degree 1, but the converse does not always hold. As shown in \cite{C-TM} for $d=2,3,4$ there exist del Pezzo surfaces of degree $d$ over a field of cohomological dimension $1$ which do not have a zero cycle of degree 1. Or, as presented in \cite[Theorem 5.1]{CT}, a hypersurface whose index $I(X)=p$, for a prime $p\geq 5$. 

The study of zero cycles on a smooth projective variety $X$ has played an important role in algebraic geometry. For example, if $C/k$ is a smooth projective curve with a $k$ rational point, then the Chow motive $h(C)$ admits an integral Chow-Künneth decomposition, see \cite[Chapter 2]{MNP}. In general, if $X$ is a smooth projective variety with a zero cycle of degree 1, then the integral Chow motive decomposes as $h(X)=h^0(X)\oplus h^+(X) \oplus h^{2d}(X)$ with $h^0(X)\simeq \mathbb{L}$ and $h^{2d}(X)\simeq \mathbb{L}^d$ where $\mathbb{L}$ is the Lefschetz motive.

Another important fact about zero cycles concerns birational invariance, i.e. if $f:X \to Y$ is a birational map between smooth varieties over a field $k$, then $\CH_0(X)\simeq \CH_0(Y)$. We can say even more: if $f:X \to Y$ is stably birational, i.e. there exist $r,s \in \N$ such that $X\times_k \Pro^r_k\to Y\times_k \Pro^s_k$ is birational, then $\CH_0(X)\simeq \CH_0(Y)$. This gives another tool for studying rationality problems and leads to the notions of $\CH_0$-universal triviality, which in the Bloch-Srinivas case induces a decomposition of the diagonal as in \cite[Proposition 1]{BS}, \cite[Lemma 1.3]{ACTP} and \cite[Proposition 1.4]{CTP}.

For a smooth projective variety $X$ of dimension $d$ over a field $k$ we define the group $\CH_0^L(X)$ as follows:
\begin{align*}
    \CH_0^L(X) := \CH^d_L(X) = H^{2d}_L(X,\Z(d)).
\end{align*}
In the present article we focus on the study of some properties of the group $\CH_0^L(X)$ (or $\CH^\et_0(X)$ after inverting the characteristic of $k$), looking at it with the purpose of obtaining a refinement of classical facts such as the birational invariance property and the existence of Lichtenbaum zero cycles of degree 1 which will induce a decomposition of the diagonal. 

For a smooth projective variety $X$,  we define \textit{\'etale degree map}, denoted by $\deg_\et$, as the push-forward of $g:X\to \spc(k)$ using the category $\text{DM}_\et(k,\Z)$, and the \textit{\'etale index of} $X$ as an analogue of $I(X)$, as follows:
\begin{align*}
    I_\et(X):= \gcd\left\{\deg_\et(\CH_0^\et(X))\cap \Z\right\}.
\end{align*}
The main results of this article concern the existence of smooth and projective varieties $X$ over a field of cohomological dimension $\leq 1$ whose index $I(X)>1$ but $I_\et(X)=1$, as the following theorems show:

\begin{theo}[Theorem \ref{teo1}]
There exists a smooth projective surface $S$ over a field $k$, with $\text{char}(k)=0$ of cohomological dimension $\leq 1$, without zero cycles of degree one but $I_\et(S)=1$.   
\end{theo}

\begin{theo}[Theorem \ref{teo2}]
    For each prime $p\geq 5$ there exists a field $k$ such that $\text{char}(k)=0$ with $\text{cd}(k)=1$ and a smooth projective hypersurface $X \subset \Pro^p_k$ with $I_\et(X)=1$ but index $I(X)=p$.
\end{theo}

To find this kind of varieties, we use the Proposition \ref{propo} which characterises some smooth varieties $X$ over a field $k$ of cohomological dimension $\leq 1$, those such that $\text{Alb}(X_{\bar{k}})_\text{tors}=0$, whose \'etale degree map is surjective. The proof is based on the fact that the condition $\text{Alb}(X_{\bar{k}})_\text{tors}=0$ implies that $\CH_0^L(X_{\bar{k}})_{\text{hom}}$ is uniquely divisible, i.e. with trivial Galois cohomology in positive degrees. After this, we note that the varieties presented in \cite[Th\'eor\`eme 1.1]{C-TM}, \cite[Th\'eor\`eme 1.2]{C-TM} and \cite[Theorem 5.1]{CT}  satisfy the hypothesis of Proposition \ref{propo}.

These results give us the first refinement for the existence of $h_\et(X)=h_\et^0(X)\oplus h_\et^+(X)\oplus h_\et^{2d}(X)$ in the category of integral \'etale motives but not in the category of integral Chow motives. Despite this new refinement of the index of a smooth projective variety, we give an example of how the property $I_\et(X)=1$ is not always achieved. For Severi-Brauer varieties $X$ we show that $I_\et(X)$ is greater than or equal to the order of the class $[X] \in \text{Br}(k)$ as follows:
\begin{theo}[Theorem \ref{teoSV}]
    Let $X$ be a Severi-Brauer variety of dimension $d$ over a field $k$. Then the image of $\deg_\et:CH^d_\et(X)\to \Z$ is isomorphic to a subgroup of $\text{Pic}(X)$ and in particular $I_\et(X)\geq \text{ord}([X])$ where $[X]$ is the Brauer class of $X$ in $\text{Br}(k)$. Moreover, if $\text{cd}(k)\leq 4$, then the group $\CH_0^L(X)$  fits in the following exact sequence
    \begin{align*}
         0 \to E^{3,2d-1}_\infty(d) \to \CH^d_L(X) \to E^{0,2d}_\infty(d) \to 0.
    \end{align*}
with $E^{0,2d}_\infty(d)= \text{ker}\left\{\CH^d_\et(\mathbb{P}^d_{\bar{k}})^{G_k}\to \text{Br}(k)\right\}$ and in particular $I_\et(X)=\text{ord}([X])$.
\end{theo}

We then prove that this bound also holds for the product of Severi-Brauer varieties. To prove this, we give the following generalisation of \cite[Theorem 5.4.10]{GS06}:

\begin{lemma*}[Lemma \ref{lemBr}]
    Let $X$ be a Severi-Brauer variety of dimension $d$ over a field $k$. For the product $X^{\times n}:=\overbrace{X\times \ldots \times X}^{n-\text{times}}$ we then obtain an exact sequence
 \begin{align*}
    0 \to \text{Pic}(X^{\times n}) \to \text{Pic}(\Pro^d_{\bar{k}} \times \ldots \times \Pro^d_{\bar{k}})^{G_k} \simeq  \Z\oplus \ldots \oplus \Z \xrightarrow{s} \text{Br}(k) \to \text{Br}(X^{\times n})
\end{align*}
where $s$ sends $(a_1,\ldots,a_n) \mapsto \sum_{i=1}^n a_i\left[X\right] \in \text{Br}(k)$.
\end{lemma*}

With this lemma, we can state and prove the following result for a product of Severi-Brauer varieties:

\begin{theo}[Theorem \ref{teo}]
    Let $k$ be a field and let $X$ be a Severi-Brauer variety over $k$ of dimension $d$. Then $I_\et(X^{\times n}) \geq I_\et(X) \geq \text{ord}([X])$.
\end{theo}

With this goal in mind, we will start by recalling some properties of Lichtenbaum's \'etale motivic cohomology, giving in Lemma \ref{lemvan} vanishing results depending on the \'etale cohomological dimension of the variety in question. Then we continue by giving in Lemma \ref{proHS} the Hochschild-Serre spectral sequence for Lichtenbaum's cohomology, for example, for a finite \'etale morphism $f:Y \to X$ with Galois group $G$ and for each degree $n\in \Z$ we have a spectral sequence
$$E^{r,s}_2(n) = H^r(G,H_L^s(Y,\Z(n)))\Longrightarrow H^{r+s}_L(X,\Z(n)).$$
In addition, in Lemma 2.8 we give a proof of classical formulae, such as those for projective bundles, smooth centre blow-up and varieties admitting cellular decomposition for \'etale motivic and Lichtenbaum cohomology. These formulae and the vanishing lemma for Lichtenbaum cohomology play an important role in the study of the birationality properties of the zero cycles of Lichtenbaum cohomology, giving counterexamples where stability and birationality fail for $\CH_0^\et(X)$.

Then we define the \'etale degree map, denoted $\deg_\et$, as the push-forward of the structural morphism $g: X\to \spc(k)$, in other words $\deg_\et:\CH_0^\et(X) \to \Z[1/p]$, where $p$ is the exponential characteristic of the base field $k$. One of the main tools we use to obtain the results is that $\deg_\et$ factors through the term $E^{0,2d}_\infty(d)$ given by the Hochschild-Serre spectral sequence. This is quite important to give a description of the nature of $\CH_0^L(X)$ since $E^{0,2d}_\infty(d)\hookrightarrow E^{0,2d}_2(d)=\CH^d(X_{\bar{k}})[1/p]^{G_k}$, in other words it is a subgroup of the zero cycles in $X_{\bar{k}}$ which are Galois invariant. 

This article is organised as follows: in section 2 we present the preliminaries. In the beginning, we present the definition of profinite cohomology and a special case of it, as \textit{Galois cohomology}. We start by defining \'etale motivic and Lichtenbaum cohomology as an analogue of motivic cohomology, and state some results that we will use throughout the present article. Then we construct the Hochschild-Serre spectral sequence for Lichtenbaum cohomology. After that we discuss and recall the formulas of a projective vector bundle, blow-up with smooth center and varieties $X$ which admit cellular decomposition, for motivic, \'etale motivic and Lichtenbaum cohomology.

Section 3 is devoted to the discussion of birational invariance of the $\CH_0^L(X)$ presenting cases where it is known to be a birational invariant and showing that this is not always the case for some non-algebraically closed fields. 

In Sections 4 and 5 we study the group of zero cycles for \'etale Chow groups. In Section 4 we give the definition of the \'etale degree map and relate the filtration on $\CH_0^\et(X)$ induced by the Hochschild-Serre spectral sequence and its factorisation by the term $E^{0,2d}_\infty(d)$, i.e. a subgroup of the fixed points of the Galois action on $\CH_0(X_{\bar{k}})$.

Section 5 deals with the examples of varieties $X$ over a field $k$ with $I_\et(X)=1$ but $I(X)>1$.  For this, we present Proposition \ref{propo}, which is a general statement characterising some of the varieties $X$ whose \'etale degree is equal to 1. We will then apply Proposition \ref{propo} to the examples given in  
\cite{C-TM} and \cite[Theorem 5.1]{CT} to obtain Theorem \ref{teo1} and \ref{teo2}. Then we move to varieties $X$ which do not have $I_\et(X)$ equals to one. We give the bound for such \'etale degree of Severi-Brauer variety as in Theorem \ref{teoSV} and then we continue with the product of a Severi-Brauer varieties. To achieve this, we present the generalisation lemma \ref{lemBr} and then in Theorem \ref{teo} we obtain the bound for $I_\et(X^n)$.

\section*{Conventions}

Let $k$ be a field, we denote as $k^{\text{sep}}$ and $\Bar{k}$ the separable and algebraic closure of $k$ respectively. For a prime number $\ell$, we denote the $\ell-$cohomological dimension of $k$ as $\text{cd}_\ell(k)$, and we set the cohomological dimension of  $k$ to be $\text{cd}(k):=\sup_{\ell}\left\{\text{cd}_\ell(k)\right\}$. Let $G$ be an abelian group, $\ell$ a prime number and $r\geq 1$, then we denote  $G[\ell^r]:=\left\{g \in G \ | \ \ell^r\cdot g =0 \right\}$, $G\{\ell\}:=\bigcup_{r} G[\ell^r]$, $G_\text{tors}$ denotes the torsion subgroup of $G$. Continuing with the same hypothesis for $G$, for an integer $p$, we set $G[1/p]:=G \otimes_{\Z} \Z[1/p]$. The prefix ``L-" indicates the respective version of some result, conjecture, group, etc... in the Lichtenbaum setting. If now $G$ is a profinite group, i.e. can be written as $G=\varprojlim G_i$ with $G_i$ finite groups, and $A$ is a $G$-module we will consider its cohomology group $H^j(G,A)$ as the continuous cohomology group of $G$ with coefficients in $A$ defined as $H^j(G,A):=\varinjlim H^j(G_i,A^{H_i})$ with $H_i$ running over the open normal subgroups of $G$ such that $G/H_i \simeq G_i$. $\text{Sm}_k$ will denote the category of smooth schemes over $k$ and $X_\et$ denotes the small \'etale site of $X$.

\section*{Acknowledgements}

The author thanks his advisors Fr\'ed\'eric D\'eglise and Johannes Nagel for their suggestions, useful discussions and time spent reading this article. He also thanks Jean-Louis Colliot-Thélène for pointing out the reference \cite{ELW} and the example given in Remark \ref{examC}.(2). Finally, he would like to
thank the referee for useful comments. This work was supported by the EIPHI Graduate School (contract ANR-17-EURE-0002) and the FEDER/EUR-EiPhi Project EITAG. We thank the French “Investissements d’Avenir” project ISITE-BFC (ANR-15-IDEX-0008) and the French ANR project “HQ-DIAG” (ANR-21-CE40-0015). 
 
\section{Preliminaries}

When considering a profinite group $G$, defined as the inverse limit of finite groups $G_i$, it is useful to recall a fact about continuous cohomology with coefficients in a uniquely divisible module. This fact is a direct consequence of \cite[Proposition 6.1.10]{Wei}, and will be used several times throughout this article.

\begin{lemma}
    Let $G$ be a profinite commutative group and let $A$ be a $G$-module which is uniquely divisible. Then $H^n(G,A)=0$ for all $n\geq 1$.
\end{lemma}


Let $k$ be a field, fix a separable closure denoted by $k^{\text{sep}}$ and denote by $G_k$ its Galois group. Our main interest is to study the cohomology of the group $G_k$. For a finite Galois extension $K/k$ we denote by $\text{Gal}(K/k)$ the Galois group of $K$ and recall that $G_k\simeq \varprojlim \text{Gal}(K/k)$ where $K$ runs through the finite Galois extensions of $k$, so it is a profinite group. The importance of this fact throughout the paper is reflected in the relationship between Galois cohomology and Lichtenbaum cohomology groups through a Hoschschild-Serre spectral sequence. 

Let us now give a brief overview of \'etale motivic cohomology and its most used properties in this article. In this subsection we will use the category of \'etale motives, since we won't mention much more details about the construction and/or functorial behaviour of the category; for more details about these properties we refer the reader to \cite{Ayo} and \cite{CD16}. Let $k$ be a perfect field and $R$ be a commutative ring. We denote the category of effective motivic \'etale sheaves with coefficients in $R$ over the field $k$ by $\text{DM}^{\eff}_\et(k,R)$, and if we invert the Lefschetz motive, we get the category of motivic \'etale sheaves with coefficients in $R$ denoted by $\text{DM}_\et(k,R)$. One defines the \textbf{\'etale motivic cohomology} group of bi-degree $(m,n)$ with coefficients in a commutative ring $R$ as
\begin{align*}
H_{M,\et}^m(X,R(n)):= \ho_{\text{DM}_\et(k,R)}(M_\et(X),R(n)[m]).
\end{align*}
where $M_\et(X)=\rho^* M(X)$ with $\rho$ is the canonical map associated to the change of topology $\rho:\left(\text{Sm}_k\right)_\et\to \left(\text{Sm}_k\right)_{\text{Nis}}$ which induces an adjunction $\rho^*:= \mathbf{L}\rho^*:\text{DM}(k,\Z)\rightleftarrows  \text{DM}_\et(k,\Z):\mathbf{R}\rho_*=:\rho_*$. In particular we define the \textbf{\'etale Chow groups} of codimension $n$ as the \'etale motivic cohomology in bi-degree $(2n,n)$ with coefficients in $\Z$, i.e. 
\begin{align*}
\text{CH}_\et^n(X):&=H_{M,\et}^{2n}(X,\Z(n))\\
&=\ho_{\text{DM}_\et(k,\Z)}(M_\et(X),\Z(n)[2n]).
\end{align*}

\begin{remark}
\begin{enumerate}
    \item Let $k$ be a field, let $\ell$ be a prime number different from the characteristic of $k$ and let $r\in \N$. By the rigidity theorem for torsion motives, see \cite[Theorem 4.5.2]{CD16}, we have an isomorphism 
    $$H_{M,\et}^m(X,\Z/\ell^r(n))\simeq H_\et^m(X,\mu_{\ell^r}^{\otimes n}).$$
    \item Note that $M_\et(X)$ (and also $M(X)$) can be defined even if $X$ is singular, but for simplicity in this paper we consider $X$ to be smooth.
\end{enumerate}
\end{remark}

We consider a second notion of the \'etale version of Chow groups, namely the well-known Lichtenbaum cohomology groups, groups defined by the hypercohomology of the \'etale sheafification of Bloch's complex sheaf. These groups are characterised by Rosenschon and Srinivas in \cite{RS} using \'etale hypercoverings. In this context we consider $\text{Sm}_k$ as the category of smooth separated $k$-schemes over a field $k$. We denote by $z^n(X,\bullet)$ the cycle complex of abelian groups defined by Bloch 
\begin{align*}
   z^n(X,\bullet): \cdots \to z^n(X,i) \to \cdots \to z^n(X,1)\to z^n(X,0) \to 0 
\end{align*}
where the differentials are given by the alternating sum of the pullbacks of the face maps and their homology groups define the higher Chow groups $\text{CH}^n(X,m)=H_m(z^n(X,\bullet))$.

Let us recall that  $z^n(X,i)$ and the complex $z^n(X,\bullet)$ are covariant functorial for proper maps and contravariant functorial for flat morphisms between smooth $k$-schemes, see \cite[Proposition 1.3]{Blo10}, therefore for a topology $t \in \left\{\text{fppf}, \ \et, \ \text{Nis},  \ \text{Zar} \right\}$ we have a complex of $t$-presheaves
$z^n(-,\bullet):U \mapsto z^n(U,\bullet)$. In particular the presheaf $z^n(-,i):U \mapsto z^n(U,i)$ is a sheaf for $t \in \left\{\text{fppf}, \ \et, \ \text{Nis},  \ \text{Zar} \right\}$, see \cite[Lemma 3.1]{Ge04}, and then $z^n(-,\bullet)$ is a complex of sheaves for the small \'etale, Nisnevich and Zariski sites of $X$. We set the complex of $t$-sheaves 
\begin{align*}
R_X(n)_t = \left(z^n(-,\bullet)_t \otimes R \right)[-2n]
\end{align*}
where $R$ is an abelian group and for our purposes we just consider $t= \text{Zar}$ or $\et$ and then we compute the hypercohomology groups $\mathbb{H}^m_t(X,R_X(n)_t)$. For example, setting $t=$ Zar and $R=\Z$ the hypercohomology of the complex allows us to recover the higher Chow groups $\text{CH}^n(X,2n-m)\simeq \mathbb{H}_{\text{Zar}}^m(X,\Z(n))$. We denote the motivic and Lichtenbaum cohomology groups with coefficients in $R$ as
\begin{align*}
H_M^m(X,R(n))= \mathbb{H}_\text{Zar}^m(X,R(n)), \quad H_L^m(X,R(n))=\mathbb{H}_\et^m(X,R(n))
\end{align*}
and in particular we set $\text{CH}_L^n(X):=H^{2n}_L(X,\Z(n))$. Let $\rho: X_\et\to X_{\text{Zar}}$ be the canonical morphism of sites, then the associated adjunction formula $\Z_X(n)\to R \rho_* \rho^* \Z_X(n)= R \rho_* \Z_X(n)_\et$ induces \textit{comparison morphisms}
\begin{align*}
\text{H}_M^m(X,\Z(n)) \xrightarrow{\kappa^{m,n}} \text{H}_L^m(X,\Z(n))
\end{align*}
for all bi-degrees $(m,n) \in \Z^2$. We can say more about the comparison map: due to \cite[Theorem 6.18]{VV}, the map  $\kappa^{m,n}:H^m_M(X,\Z(n))\to H^m_L(X,\Z(n))$ is an isomorphism for $m\leq n+1$ and a monomorphism for $m \leq n+2$.

If $R$ is torsion then we can compute the Lichtenbaum cohomology as an \'etale cohomology. To be more precise for a prime number $\ell$, $r\in \mathbb{N}\geq 1$ and $R=\Z/\ell^r$ then we have the following quasi-isomorphisms
\begin{align*}
    (\Z/\ell^r)_X(n)_\et \xrightarrow{\sim} \begin{cases}
      \mu_{\ell^r}^{\otimes n} &\text{if char}(k)\neq \ell \\
      \nu_r(n)[-n] &\text{if char}(k)= \ell
    \end{cases}
\end{align*}
where $\nu_r(n)$ is the logarithmic de Rham-Witt sheaf. After passing to direct limit we have also quasi-isomorphisms
\begin{align*}
    (\Q_\ell/\Z_\ell)_X(n)_\et \xrightarrow{\sim} \begin{cases}
      \varinjlim_{r}\mu_{\ell^r}^{\otimes n} &\text{if char}(k)\neq \ell \\
      \varinjlim_{r}\nu_r(n)[-n] &\text{if char}(k)= \ell
    \end{cases}
\end{align*}
and finally set $ (\Q/\Z)_X(n)_\et =\bigoplus (\Q_\ell/\Z_\ell)_X(n)_\et\xrightarrow{\sim } \Q/\Z(n)_\et$. In the case when $k=\bar{k}$ then for a smooth projective variety $X$ and  $n\geq \text{dim}(X)$ by the Suslin rigidity theorem,  the morphism $\Z_X(n)\to R\rho_* \Z_X(n)_\et$ is a quasi-isomorphism. For this, see \cite[Section 6, Theo. 4.2]{V} and  \cite[Section 2]{Geis1}, and for a proof we refer to \cite[Lemma 2.2.2]{ros}. Another important reminder concerns the vanishing of higher Chow groups. Following \cite[Theorem 3.6]{MVW} for every smooth scheme and any abelian group $R$, we have $H^m_M(X,R(n))=0$ when $m>n+\dim(X)$. Also we have a second vanishing theorem for motivic cohomology, presented in \cite[Theorem 19.2]{MVW}, for $X$ and $R$ under the same assumptions as before, we have that  $H^m_M(X,R(n))=0$ when $m>2n$.

\begin{remark}
Let $k=\bar{k}$. Since the map $\Z_k(n)\to R\rho_* \Z_k(n)_\et$ is a quasi-isomorphism for all $n \geq 0$ we obtain that  $H^m_L(\spc(k),\Z(n))\simeq H^m_M(\spc(k),\Z(n))$ for all $(m,n) \in \Z \times \N$. In particular $H^m_L(\spc(\bar{k}),\Z(n))=0$ for $m>n\geq 0$.
\end{remark}

By pursuing a similar vanishing theorem for Lichtenbaum cohomology is that we obtain the following results about the vanishing of the cohomology groups:

\begin{lemma}\label{lemvan}
Let $k$ be a field and let $X$ be in $\text{SmProj}_k$. Consider a bi-degree $(m,n) \in \Z^2$ we then have the following: 
\begin{enumerate}
    \item if $m>n$ and $m>\text{cd}(k)+1$ we have that $H^m_L(\spc(k),\Z[1/p](n))=0$.
    \item More generally if $m>n+\text{cd}(X)$ then $H^m_L(X,\Z(n))=0$.
\end{enumerate}
\end{lemma}

\begin{proof}
    Statement (1) is a direct consequence of \cite[Theorem 6.18]{VV} and the isomorphism $H^m_L(k,\Z(n)) \simeq H^{m-1}_\et(k,\Q/\Z(n))$ if $m>n$.
    
For the more general case presented in (2), let $X$ be $\text{SmProj}_k$ and consider the motivic complex $\Z(n)$. The complex vanishes for degrees greater than $n$. Consider the canonical map $\rho: X_\et \to X_{\text{Zar}} $, the functor induced by the change as
\begin{align*}
    \rho^*:D(\text{AbShv}_{\text{Zar}}(\text{Sm}_k)) \leftrightarrows D(\text{AbShv}_\et(\text{Sm}_k)):R\rho_*.
\end{align*}
Recall that $H^m_L(X,\Z(n))$ is the hypercohomology of the complex of \'etale sheaves $\Z_X(n)_\et$. Since the functor $\rho^*$ is exact, the \'etale cohomology sheaves of $\Z_X(n)_\et$ vanish in cohomological degree $>n$. So we conclude that $H^m_L(X,\Z(n))=0$ for $m>n+\text{cd}(X)$.
\end{proof}

Let us denote the Suslin-Voevodsky motivic complex of Nisnevich sheaves in $\text{Sm}_k$ as $\Z_{SV}(n)$. Since $\Z_X(n)_\et \xrightarrow{\sim} \Z_{SV}(n) \Big|_{X_\et} $ is a quasi-isomorphism then we have a comparison map
\begin{align*}
\rho^{m,n}:H^m_L(X,\Z(n)) \to H^m_{M,\et}(X,\Z(n))
\end{align*}
which is induced by the quasi-isomorphism $\Z_X(n)_\et \xrightarrow{\sim} \Z_{SV}(n) \Big|_{X_\et} $ and $\Z_{SV}(n)_\et \to L_{\mathbb{A}^1}(\Z_{SV}(n)_\et)$ where $ L_{\mathbb{A}^1}$ is the $\mathbb{A}^1-$localisation functor of \'etale motivic complexes. According to \cite[Theorem 7.1.2]{CD16} the morphism $\rho^{m,n}$ becomes an isomorphism after inverting the characteristic exponent of $k$.  If $p$ is equal to the field characteristic, then by using $\Z[1/p]_X(n)_\et$ we can recover the functorial properties of \'etale motivic cohomology for Lichtenbaum cohomology.

The latter isomorphism after inverting the exponential characteristic of the field gives us an important tool for studying \'etale motivic cohomology, which is the relationship between Galois cohomology and Lichtenbaum cohomology groups via the Hochschild-Serre spectral sequence for Lichtenbaum cohomology. This was stated in \cite{CTK} and a proof was given in  \cite[Pages 6-7]{RS18}:

\begin{lemma}{\cite[Page 31]{CTK}}\label{proHS}
Let $p:Y \to X$ be a finite Galois covering of $X$ with Galois group $G$, then there exists a convergent Hochschild-Serre spectral sequence with abutment the Lichtenbaum cohomology group 
\begin{align}\label{sshs}
E^{r,s}_2(n) = H^r(G,H_L^s(Y,\Z(n)))\Longrightarrow H^{r+s}_L(X,\Z(n)).
\end{align}
\end{lemma}

\begin{remark}
    Let $k$ be a field and $k^{\text{sep}}$ be a separable closure. Since cohomology commutes with inverse limits, and considering that the absolute Galois group of $k$ is defined as the inverse limit over the finite separable field extensions  $G_k=\varprojlim_{k\subset K \subset k^{\text{sep}}} \text{Gal}(K/k)$, then again for $[K:k]<\infty $ we have a convergent spectral sequence $H^r(\text{Gal}(K/k),H^s_L(X_K,\Z(n))) \Longrightarrow H^{r+s}_L(X,\Z(n))$. Mixing the compatibility of hypercohomology with inverse limits we obtain a spectral sequence for the absolute Galois group 
    \begin{align}\label{spectralfield}
       E_2^{r,s}(n)=H^r(G_k,H^s_L(X_{k^{\text{sep}}},\Z(n))) \Longrightarrow H^{r+s}_L(X,\Z(n)). 
    \end{align}
\end{remark}

In the following we recall some facts about the structure of the Lichtenbaum cohomology group of smooth projective varieties over an algebraically closed field. For more details on the structure and properties of Lichtenbaum cohomology we refer the reader to  \cite[Proposition 4.17]{Kahn}, \cite[Theorem 1.1]{Geis} and \cite[Theorem 3.1]{RS}. Consider $X \in \text{SmProj}_k$ with $k = \bar{k}$ of characteristic exponent $p$ and consider a bi-degree $(m,n) \in \Z^2$. If $m\neq 2n$ then according to \cite[Theorem 3.1]{RS} $H^m_L(X,\Z(n))\otimes \Q_\ell/\Z_\ell=0$ for all $\ell \neq p$. Denoting $(\Q/\Z)'=\bigoplus_{\ell\neq p} \Q_\ell/\Z_\ell$ we have that $H^m_L(X,\Z[1/p](n))\otimes (\Q/\Z)'=0$ and then
\begin{align*}
    0\to H^m_L(X,\Z(n))_\text{tors} \to H^m_L(X,\Z(n)) \to H^m_L(X,\Z(n)) \otimes \Q \to 0.
\end{align*}
In fact this short exact sequence splits, so for $m\neq 2n$, $H^m_L(X,\Z(n))$ is the direct sum of a uniquely divisible group and a torsion group. For the case when $m\neq 2n+1$ we have an isomorphism $H^m_L(X,\Z(n))\{\ell\}\simeq H^{m-1}_\et(X,\Q_\ell/\Z_\ell(n))$ again considering $\ell \neq p$.

Since for any $n$ we have an exact triangle
\begin{align*}
    \Z_X(n)_\et \to \Q_X(n)_\et \to (\Q/\Z)_X(n)_\et \xrightarrow{+1}
\end{align*}
and for $m<0$ the group $H^m_\et(X,\Q/\Z(n))$ vanishes, then we conclude that for such $m$ we have isomorphisms $H^m_L(X,\Z(n))\simeq H^m_L(X,\Q(n))$ i.e. the Lichtenbaum cohomology groups with integral coefficients are $\Q$-vector spaces, thus uniquely divisible groups.

Now let us return to the Hochschild-Serre spectral sequence for Lichtenbaum cohomology. Suppose $X$ is a smooth projective geometrically integral $k$-variation of dimension $d$ with $k$ a perfect field of characteristic exponent $p$, and let $\bar{k}$ be an algebraic closure of $k$ with Galois group $G_k$ and denote $X_{\bar{k}}=X\times_{\spc(k)}\spc(\bar{k})$. For such $X$ we consider a special case of the Hochschild-Serre spectral sequence (\ref{sshs})
\begin{align}\label{spectralfieldf}
    E^{r,s}_2(n):= H^r(G_k,H^s_L(X_{\bar{k}},\Z[1/p](n))) \Longrightarrow H^{r+s}_L(X,\Z[1/p](n))
\end{align}
with the previous recall, we can give information about the vanishing of some terms $E^{r,s}_2(n)$ of (\ref{spectralfieldf}):
\begin{itemize}
    \item $E^{r,s}_2(n)=0$ for $r<0$ because we work with the cohomology of a profinite group.
    \item $E^{r,s}_2(n)=0$ for $r>0$ and $s<0$ by the uniquely divisibility of $H^s_L(X_{\bar{k}},\Z[1/p](n))$.
    \item $E^{r,s}_2(n)=0$ for $r>\text{cd}(k)$ and $s\neq 2n$. Indeed, as $s\neq 2n$ then 
    \begin{align*}
        H^s_L(X_{\bar{k}},\Z[1/p](n)) \simeq H^s_L(X_{\bar{k}},\Q(n)) \oplus H^s_L(X_{\bar{k}},\Z[1/p](n))_\text{tors},  
    \end{align*}
since $ H^s_L(X_{\bar{k}},\Q(n))$ is uniquely divisible, so for a pair $(r,s)$ satisfying the above restrictions, we have that     
\begin{align*}
    H^r(G_k,H^s_L(X_{\bar{k}},\Z[1/p](n)))\simeq H^r(G_k,H^s_L(X_{\bar{k}},\Z[1/p](n)))_\text{tors}.
\end{align*}
Now, if  $r>\text{cd}(k)$, the group $H^r(G_k,H^s_L(X_{\bar{k}},\Z[1/p](n)))_\text{tors}$ vanishes.
\end{itemize}

\begin{example}\label{remdeg}
 For instance if we assume that $\text{cd}(k)\leq 2$ and $s<2n$, then we have the following isomorphisms

\begin{align*}
E^{0,s}_\infty(n) &=\text{ker}\left\{d_2:E^{0,s}_2(n)\to E^{2,s-1}_2(n)\right\} \\
&= \text{ker}\left\{d_2:H^s_L(X_{\bar{k}},\Z(n))^{G_k} \to H^2(G_k,H_L^{s-1}(X_{\bar{k}},\Z(n)))\right\} \\
E^{1,s}_\infty(n)&\simeq E^{1,s}_2(n)\\
E^{2,s}_\infty(n)&\simeq E^{2,s}_2(n)/\text{im}\left\{E^{0,s+1}_2(n)\to E^{2,s}_2(n) \right\}\\
&=H^2(G_k,H^s_L(X_{\bar{k}},\Z(n)))/\text{im}\left\{H^{s+1}_L(X_{\bar{k}},\Z(n))^{G_k}\to H^2(G_k,H^s_L(X_{\bar{k}},\Z(n))) \right\}. \\
\end{align*}
\end{example}

We conclude this section by mentioning some well-known results about the structure of \'etale motivic and Lichtenbaum cohomology groups of projective bundles, smooth blow-ups and varieties with cellular decomposition:

\begin{lemma}\label{lemme}
Let $k$ be a field  of characteristic $p\geq 0$ and let $X$ be a smooth projective scheme over $k$. Let $\varepsilon \in \left\{M, L,(M,\et)\right\}$ and consider a bi-degree $(m,n) \in \Z^2$, then there exists the following characterisations:
\begin{enumerate}
    \item[(i)] If $r\geq 0$ and let  $\mathbb{P}_X^r$ be the projective space of dimension $r$ over $X$, then the canonical map $\Pro_X^r \to X$ induces an isomorphism:
        \begin{align}\label{proBun}
            H^m_\varepsilon(\mathbb{P}^r_X,\Z(n))\simeq \bigoplus_{i=0}^r H^{m-2i}_\varepsilon(X,\Z(n-i)).
        \end{align}
    \item[(ii)] Let $Z$ be a smooth projective sub-scheme of $X$ of codimension $c\geq 2$. Denote the blow-up of $X$ along $Z$ as $\text{Bl}_Z(X)$, then
        \begin{align}\label{blowup}
            H^m_\varepsilon(\text{Bl}_Z(X),\Z(n))\simeq H^m_\varepsilon(X,\Z(n))\oplus \bigoplus_{i=1}^{c-1}H_\varepsilon^{m-2i}(Z,\Z(n-i)). 
        \end{align}
        \item[(iii)]  Assume that a map $f:X\to S$ which is a flat of relative dimension $r$ over a smooth base $S$. Assume as well that $X$ has a filtration $X=X_t \supset X_{t-1} \supset \ldots \supset X_0 \supset X_{-1}=\emptyset$ where $X_i$ is smooth and projective for all $i$ and $ U_i:=X_i-X_{i-1}\simeq \mathbb{A}^{r-d_i}_S$ then we obtain the following formula:
        \begin{align*}
            H^m_{M,\et}(X,\Z(n))\simeq \bigoplus_{i=0}^{t} H^{m-2d_i}_{M,\et}(S,\Z(n-d_i)).
        \end{align*}
\end{enumerate}
\end{lemma}
\begin{proof}
The statements (i) and (ii) are obtained in similar ways: first notice that by properties of $\text{DM}(k,R)$ with $R$ a commutative ring, see \cite[Section 14 \& 15]{MVW}, we have canonical isomorphisms of motives 
\begin{align*}
    \bigoplus_{i=0}^r M(X)(i)[2i] \xrightarrow{\simeq} M(\Pro_X^r)\hspace{4mm}\text{and}\hspace{4mm}     M(\text{Bl}_Z(X))\simeq M(X) \oplus \left(\bigoplus_{i=1}^{c-1} M(Z)(i)[2i]\right),
\end{align*}
thus the statments hold when $\varepsilon=M$ and $\varepsilon=(M,\et)$.

When $\varepsilon=L$ both formulas (1) holds because for $R=\Q$ we recover the formulas for rational coefficients whereas for finite coefficients we invoke \cite[VI, Lemma 10.2]{Mil80} when $\ell^r\neq p$ and \cite[I, Th\'eor\`eme 2.1.11]{Gros} for the logarithmic Hodge-Witt complex. The formula (2) holds again because it holds for $R=\Q$ and for finite coefficients by the proper base change \cite[VI, Corollary 2.3]{Mil80} and \cite[IV, Corollaire 1.3.6]{Gros} for the logarithmic Hodge-Witt complex.

Meanwhile for $\varepsilon= (M,\et)$ this holds because of the previous isomorphisms when $R=\Z$ and the fact that the functor $\rho^*:\text{DM}(k,\Z)\to \text{DM}_\et(k,\Z)$ is exact.

To prove (3), one uses the homotopy invariance in $\text{DM}_\et(k,\Z)$ which gives the homotopy invariance of étale motivic cohomology, together with the arguments given in \cite[Appendix]{kock}.
    
\end{proof}

\begin{example}
    By Lemma \ref{lemme}, the Lichtenbaum cohomology groups of the projective  space over a field $k$ are the following 
\begin{align*}
    \CH^m_L(\mathbb{P}^l_k)\simeq \bigoplus_{j=0}^{m}\CH^j_L(\spc(k)).
\end{align*}
For $i\geq 2$ we have that $\CH^i_L(\spc(k))\simeq \mathbb{H}^{2i}_\text{Zar}(\spc(k),\tau_{\geq i+2}R\pi_*\Z(i)_\et)$. Meanwhile, by the vanishing of motivic cohomology, we have that $\CH^{j}(\spc(k))\simeq 0$ for $j\geq 1$, thus we obtain that $\CH^i_L(\spc(k))\simeq H^{2i-1}_\et(\spc(k),\Q/\Z(i))$. 
\end{example}

\section{Birational invariance}

Let us recall some definitions of birational geometry. Let $X, Y$ be smooth $k$ varieties, we say that a rational map $f:X \to Y$ is birational if there exist open subsets $U \subset X$ and $V \subset Y$ such that $f:U \to V$ is an isomorphism. We say that $X$ is stably birational to $Y$ if there exist $r, s \in \N$ such that $X\times \Pro^r_k \to Y\times \Pro^s_k$ is a birational morphism. The importance of $\CH_0(X)$ lies in its birational invariance, for which we refer to \cite[Example 16.1.11]{Fulton}. Now suppose $X \to Y$ is stably birational, then there exist $r,s \in \N$ such that $X\times \Pro^r_k \to Y\times \Pro^s_k$ is birational, and since $\CH_0$ is a birational invariant, we get an isomorphism
\begin{align*}
    \CH_0(X\times \Pro^r_k)\xrightarrow{\simeq} \CH_0(Y\times \Pro^s_k),
\end{align*}
but by the projective bundle formula for Chow groups and the vanishing properties we get that $\CH_0(X\times \Pro^r_k)\simeq \CH_0(X)$ and $\CH_0(Y\times \Pro^s_k)\simeq \CH_0(Y)$ so $\CH_0(X)\simeq \CH_0(Y)$. So $\CH_0$ is also a stable birational invariant. 

\begin{remark}
The reference \cite[Example 16.1.11]{Fulton} gives the birational invariance of $\CH_0$ when the base field is algebraically closed, but in general the same argument works for any field.
\end{remark}

The first question that arises is whether or not $\CH_0^L(X)$ (or $\CH_0^\et(X)$) is a birational invariant or a stably birational invariant. Let $X$ be a smooth projective variety over a field $k$, because of the comparison map $\CH_0(X)\to \CH_0^L(X)$ we can say a few words about the invariance depending on the field and the dimension of $X$: if $k = \bar{k}$ then $\CH^d(X)\simeq \CH^d_L(X)$, thus we can use the stable birational invariance of 0-cycles in the classical setting quoted above, for the category $\text{SmProj}_k$. If the field is not algebraically closed, we lose many of the birational properties associated with $\CH_0$. For example, if we consider $k$ to be a field which can be embedded in $\R$ and $d\geq 2$, by invoking lemma \ref{lemme} and the vanishing properties of lemma \ref{lemvan}, we immediately see that 
\begin{align*}
    \CH^0_L(\spc(k))\neq \CH^d_L(\Pro_k^d)\simeq \bigoplus_{i=0}^d  \CH^i_L(\spc(k)).
\end{align*}

So $\CH_0^L$ is not a stable birational invariant. If we now concentrate only on the study of the birational invariance of $\CH_0^L(X)$, we have the following result:
\begin{prop}
    Let $k$ be an arbitrary field and let $X$ be a smooth projective scheme of dimension $d$ over $k$. Then $\CH_0^L$ is a birational invariant if $d \in \left\{0,1,2\right\}$.
\end{prop}

\begin{proof}
The case $d=0$ is trivial. If $d=1$, we use the isomorphism $\CH^1(X)\simeq \CH^1_L(X)$ and the birational invariance of zero cycles in the classical case. For $d=2$ we have a short exact sequence
\begin{align*}
    0 \to \CH^2(X) \to \CH^2_L(X)  \to H^3_{\text{nr}}(X,\Q/\Z(2)) \to 0.
\end{align*}
The group $\CH^2(X)$ is a birational invariant for surfaces and the unramified cohomology groups $H^3_{\text{nr}}(X,\Q/\Z(2))$ is birational invariant for any dimension, this is a consequence of the Gersten's conjecture, see \cite[Th\'eor\`eme 2.8]{CTV}, therefore $\CH^2_L(X)$ is a birational invariant. 
\end{proof}

If we go to a higher dimension, the argument with the comparison map fails. To illustrate this, consider the following: Let $X$ be a smooth projective variety of dimension three over a field $k$.  Since we have a quasi-isomorphism of complexes of Zariski sheaves over $\Z_X(3)_{\text{Zar}} \xrightarrow{\simeq} \tau_{\leq 4}R\pi_*\Z_X(3)_\et$ 
In particular, as stated in \cite[Theorem 6.6]{voez2}, one obtains the following exact triangle 
\begin{align*}
    \Z_X(3)_{\text{Zar}} \to \Z_X(3)_{\et} \to \Z_X(3)_{\et} \to \tau_{\geq 5}R\pi_*\Z_X(3)_\et \xrightarrow{+1}
\end{align*}
which induces a long exact sequence
\begin{align*}
    \to \mathbb{H}^5_\text{Zar}(X,\tau_{\geq 5}R\pi_* \Z(3)_\et) \to \CH^3(X)\to \CH^3_L(X) \to \mathbb{H}^6_\text{Zar}(X,\tau_{\geq 5}R\pi_* \Z(3)_\et) \to 0.
\end{align*}
We have that $\mathbb{H}^5_\text{Zar}(X,\tau_{\geq 5}R\pi_* \Z(3)_\et) \simeq H^4_\text{nr}(X,\Q/\Z(3))$ is a birational invariant. Therefore $\CH^3_L(X)$ is a birational invariant if and only if $\mathbb{H}^6_\text{Zar}(X,\tau_{\geq 5}R\pi_* \Z(3)_\et) $ is a birational invariant, where the latter group can be characterised by means of the hypercohomology spectral sequence (see \cite[(4.3)]{RS18})
\begin{align*}
    { }^{\text{h}}E_2^{r,s}:=H^r_{\text{Zar}}(X,R^s\tau_{\geq 5}R\pi_*\Z_X(3)_\et)\Longrightarrow H^{r+s}_{\text{Zar}}(X,\tau_{\geq 5}R\pi_*\Z_X(3)_\et).
\end{align*}
Using this spectral sequence together with \cite[Corollaire 2.8]{Kahn}, one obtains the following short exact sequence,
\begin{align*}
    0 \to H^1_\text{Zar}(X,\mathcal{H}^4_\et(\Q/\Z(3))) \to \mathbb{H}^6_\text{Zar}(X,\tau_{\geq 5}R\pi_* \Z(3)_\et) \to { }^{\text{h}}E^{0,6}_\infty \to 0
\end{align*}
where ${ }^{\text{h}}E^{0,6}_\infty = \text{ker}\left\{H^5_\text{nr}(X,\Q/\Z(3)) \to H^2_\text{Zar}(X,\mathcal{H}^4_\et(\Q/\Z(3)))\right\}$. In fact, the first counterexample can be found in dimension 3. Recall that by the formulas given in Lemma \ref{lemme} we have the following: let $X$ be a smooth projective variety and let $Z \subset X$ be a smooth subvariety of codimension c. Then for the blow-up $\tilde{X}_Z$ of $X$ along $Z$ the Lichtenbaum cohomology decomposes as follows
\begin{align*}
    \CH^d_L(\Tilde{X}_Z) \simeq \CH^d_L(X) \oplus \bigoplus_{j=1}^{c-1} \CH^{d-j}_L(Z)
\end{align*}
Notice that $d-j>d-c=\dim(Z)$, therefore the groups $\CH^{d-j}_L(Z)$ are just torsion isomorphic to $\mathbb{H}^{2(d-j)}_\text{Zar}(X,\tau_{\geq d-j+2}R\pi_*\Z(d-j)_\et)$. The next example shows how to use this fact to get the counterexample. counterexample.

\begin{example}
Firstly, if we have a non totally imaginary number field $K$ and consider an odd integer $n\geq 3$. Let $\Omega_\mathbb{R}$ be the set of real embeddings $K \hookrightarrow \mathbb{R}$. By \cite[Theorem I.4.10(c)]{milne06} or \cite[\S 6, Théorème B]{ser2} we have the morphism
\begin{align*}
    H^n(K,\mu_2)\to\bigoplus_{v \in \Omega_{\mathbb{R}}} H^n(K_v,\mu_2)
\end{align*}
is an isomorphism, where $K_v$ is the completion of $K$ with respect to the place $|\cdot|_v$. Since $K$ is an archimedean field and $v$ is a real embedding one gets $K_v\simeq \mathbb{R}$, thus $ H^n(K_v,\mu_2)\simeq  H^n(G_{\mathbb{R}},\mu_2)$ and also $G_{\mathbb{R}}$ is a cyclic group, so $H^n(G_{\mathbb{R}},\mu_2)\simeq H^1(G_{\mathbb{R}},\mu_2)\simeq \Z/2$ by \cite[Theorem 6.2.2]{Wei}. Therefore we obtain
\begin{align*}
    H^n(K,\mu_2)\simeq \bigoplus_{v \in \Omega_{\mathbb{R}}} H^n(K_v,\mu_2) \simeq \bigoplus_{v \in \Omega_{\mathbb{R}}} \Z/2.
\end{align*}

Now let $X$ be a smooth threefold with a rational point over $K$, where $K$ is an algebraic number field which is not totally imaginary, and let $Z=\spc(K)$. Let $\Tilde{X}_Z$ be the blow-up with centre $Z$, then we have
    $$\CH^3_L(\tilde{X}_Z)=\CH^3_L(X)\oplus \CH^2_L(\spc(K)).$$.
    Since $\CH^2_L(\spc(K))\simeq H^3_\et(\spc(K),\Q/\Z(2)) \neq 0$ by the previous remark, we can conclude that $\CH^3_L(\Tilde{X}_Z)\neq \CH^3_L(X)$.
\end{example}

In general, we have the proposition about the birational invariance of $\CH^L_0(X)$:
\begin{prop}
Let $k$ be a field and assume that there exists  $n \geq 2$ such that the group $H^{2n-1}_\et(\spc(k),\mu_{\ell^r}^{\otimes n}) \neq 0$ for some prime number $\ell$ and $r \in \N$, then  $\CH_0^L$ is not a birational invariant for $\text{SmProj}_k$.      
\end{prop}

\begin{proof}
 Consider the field $k$ such that $H^{2n-1}_\et(\spc(k),\mu_{\ell^r}^{\otimes n}) \neq 0$ for some prime number $\ell$, $r \in \N$ and $n\geq 2$. Let $X$ be a smooth projective variety over $k$ of dimension $d \geq n+1$ such that $X$ has a $k$ rational point. Let $\tilde{X}$ be the blow-up of $X$ along a point $Z=\spc(k) \to X$. Invoking Lemma \ref{lemme} we get
 \begin{align*}
     \CH^d_L(\tilde{X}) \simeq \CH^d_L(X) \oplus \bigoplus_{j=1}^{d-1} \CH^{d-j}_L(Z)
 \end{align*}
 As $\CH^i_L(Z)\simeq H^{2i-1}_\et(Z,\Q/\Z(i))$ for $i\geq 2$, the hypothesis implies that $\CH^n_L(Z)\neq 0$ and thus $\CH^d_L(\tilde{X}) \neq \CH^d_L(X)$.
\end{proof}

\begin{remark}
Note that the hypothesis of the previous proposition imposes the restriction that the cohomological dimension of $k$ must be at least three. Consequently, the above argument does not provide a counterexample for fields with a cohomological dimension of at most two.
\end{remark}

\section{\'Etale degree map}

Let $g:X\to \spc(k)$ be the structural morphism associated to a smooth and projective $k$-scheme of dimension $d$. Recall that the degree map is defined as 
\begin{align*}
    \deg:=g_*:\CH_0(X) \to \CH_0(\spc(k))=\Z.
\end{align*}

We can reformulate this definition due to the existence of Gysin morphisms in $\text{DM}(k,\Z)$ as is described in \cite{DegI} and \cite{DegII}. With this formalism we get the pull-back of the morphism $g$ defined as $g^*: M(\spc(k))(d)[2d]=\Z(d)[2d]\to M(X)$ in the category $\text{DM}^\eff(k,\Z)$. Using the contravariant functor $\text{Hom}_{\text{DM}^\eff(k,\Z)}(-,\Z(d)[2d])$ we get the previous definition again. From this we can extend the existence of Gysin morphisms for $\text{DM}_\et(k,\Z)$, giving us an \'etale analogue of the degree map for \'etale Chow groups:

\begin{defi}
Let $X$ be a smooth and projective scheme of dimension $d$ over $k$, where $k$ is a field of exponential characteristic equal to $p$. Then we define the \'etale degree map $\deg_\et: \text{CH}^d_\et(X)\to \CH_\et^0(\spc(k))\simeq \Z[1/p]$ as $\deg_\et:=g_*$ where $p$ is the structure morphism $g:X \to \spc(k)$. We define the \'etale index of $X$ as the greatest common divisor of the subgroup $\deg_\et(\text{CH}^d_\et(X)) \cap \Z$, denoted by $I_\et(X)$.
\end{defi}

\begin{remark}
    \begin{enumerate}
        \item Let $k$ be a field of characteristic exponent $p$. Due to functoriality properties we have the following commutative diagram
\[  
    \begin{tikzcd}
  \text{Hom}_{\text{DM}(k,\Z)}(M(Y),\Z(d)[2d])\arrow{r}{g_*} \arrow{d}{} & \text{Hom}_{\text{DM}(k,\Z)}(\Z(d)[2d],\Z(d)[2d])\arrow{d}{} \\
   \text{Hom}_{\text{DM}_\et(k,\Z)}(M_\et(Y),\Z(d)[2d])\arrow{r}{g_*} & \text{Hom}_{\text{DM}_\et(k,\Z)}(\Z(d)[2d],\Z(d)[2d])
  \end{tikzcd}
\]
where for $\CH^0_\tau(\spc(k))$ with $\tau \in \{\text{Nis},\et\}$, there are isomorphisms
\begin{align*}
 \text{Hom}_{\text{DM}(k,\Z)}(\Z(d)[2d],\Z(d)[2d])=H^{0,0}_M(\spc(k))\simeq\Z
\end{align*}
and 
\begin{align*}
 \text{Hom}_{\text{DM}_\et(k,\Z)}(\Z(d)[2d],\Z(d)[2d])=H^{0,0}_{M,\et}(\spc(k))\simeq\Z[1/p]
\end{align*}

\item By the previous point, if $\text{char}(k)=0$, $K/k$ is a finite Galois extension and $X\to \spc(k)$ is a smooth projective $k$-scheme, then the morphism $f:X_K\to X$ is a finite \'etale morphism. Since $f$ is proper, there exists an induced map $f_*:\text{CH}^d_\et(X_K)\to \text{CH}^d_\et(X)$ which fits into the following commutative diagram 
\[  
    \begin{tikzcd}
\text{CH}^d(X_K) \arrow{r}{f_*}\arrow{dd} \arrow{dr}{\deg}& \text{CH}^d(X)\arrow{dr}{\deg} \arrow[bend left=30, dashed]{dd}{} &  \\
& \Z \arrow{r}{[K:k]\cdot } & \Z \\
\text{CH}^d_\et(X_K)\arrow{r}{f_*} \arrow{ru}{\deg_\et} & \text{CH}^d_\et(X)\arrow[swap]{ur}{\deg_\et} & 
  \end{tikzcd}
\]
with $[K:k]$ the degree of the extension.

\item  It is possible to define the \'etale degree map for Lichtenbaum cohomology over a field $k=\bar{k}$. This follows because for $X$ a smooth and proper projective variety of dimension $d$ there is a quasi-isomorphism $\Z_X(n)_\text{Zar}\to R\pi_* \Z_X(n)_\et$ for $n\geq d$. In general we have to invert the characteristic exponent of $k$ and use the isomorphism between Lichtenbaum and \'etale Chow groups.
    \end{enumerate}
\end{remark}

 Let $f:X\to Y$ be a projective morphism of smooth varieties of relative dimension $c$. Again, due to the existence of Gygin morphisms in $\text{DM}_\et(k,\Z)$, we get push-forwards for \'etale motivic cohomology 
 \begin{align*}
     f_*: H^{m+2c}_{M,\et}(X,\Z(n+c)) \to  H^{m}_{M,\et}(Y,\Z(n)). 
 \end{align*}

Combining the existence of push-forward maps for \'etale motivic cohomology and the functoriality of the Hochschild-Serre spectral sequence, we get the following diagram
\[  
    \begin{tikzcd}
 H^r(G_k, H^{s+2c}_L(X_{\bar{k}},\Z[1/p](n+c)))\arrow[r,Rightarrow] \arrow{d}{\tilde{f}_*} & H^{r+s+2c}_L(X,\Z[1/p](n+c))\arrow{d}{f_*} \\
    H^r(G_k, H^{s}_L(Y_{\bar{k}},\Z[1/p](n)))\arrow[r,Rightarrow]& H^{r+s}_L(Y,\Z[1/p](n))
  \end{tikzcd}
\]
where $p$ is the exponential characteristic of $k$ and $\tilde{f}:X_{\bar{k}} \to Y_{\bar{k}}$. For the particular case of the \'etale degree map we have the following:

\begin{prop}{\label{rem}}
 Let $X$ be a smooth and projective $k$-scheme of dimension $d$ with $\text{char}(k)=p\geq 0$. Then the map $\deg_\et:\CH^d_\et(X) \to \Z[1/p]$ factors through a subgroup of $\CH^d(X_{\bar{k}})[1/p]^{G_k}$.
\end{prop}

\begin{proof}
In fact we will prove that the subgroup in question is given by the $E^{0,2d}_\infty$-term of the Hochschild-Serre spectral sequence associated to $X$. To see this, consider the structural morphism  $g:X \to \spc(k)$, then we have an induced morphism of $E_2$-terms
\begin{align*}
    E^{r,s}_2(d):= H^r(G_k,H^s_L(X_{\bar{k}},\Z[1/p](d))) \to H^r(G_k,H^{s-2d}_L(\spc(\bar{k}),\Z[1/p](0)))
\end{align*}
but as $s-2d\leq 0$ we have that 
\begin{align*}
    H^{s-2d_X}_L(\spc(\bar{k}),\Z[1/p](0)) \simeq \begin{cases}
        0 \text{ for }s\neq 2d \\
        \Z[1/p] \text{ for }s=2d
    \end{cases}
\end{align*}
that gives us $H^r_L(k,\Z[1/p](0))\simeq H^r(G_k,H^0(\bar{k},\Z[1/p](0)))$ and hence we conclude that $\deg_\et:\CH^d_\et(X)\to \Z[1/p]$ factors as
\[  
    \begin{tikzcd}
\CH^d_\et(X)\arrow[r] \arrow{dr}{\deg_\et} & E_{\infty}^{0,2d}(d)\arrow{d}{\widetilde{\deg}} \\
  & \Z[1/p]
  \end{tikzcd}
\]
where $\widetilde{\deg}$ is the composite map
\begin{align*}
    E^{0,2d}_\infty(d) \hookrightarrow E^{0,2d}_2(d) = \CH^d(X_{\bar{k}})[1/p]^{G_k} \hookrightarrow \CH^d(X_{\bar{k}})[1/p] \xrightarrow{\deg} \Z[1/p].
\end{align*}
\end{proof}

\section{Lichtenbaum zero cycles}

\subsection{Varieties where  \texorpdfstring{$I_\et(X)=1$}{Lg}}

The aim of this subsection is to construct examples where the \'etale degree map is surjective, but its classical counterpart is not. To achieve this, we start with a lemma about the divisibility of the zero cycles of degree 0 of a variety over an algebraically closed field:

\begin{lemma}
    Let $X$ be a complete scheme over an algebraically closed field $k$ of characteristic $p\geq 0$. Define $A_0(X)=\text{ker}\left\{\deg:\CH_0(X)\to \Z\right\}$, then $A_0(X)$ is a divisible group. If $X$ is a smooth quasi-projective scheme and $H^{2d-1}_\et(X,\Q_\ell/\Z_\ell(d))=0$ for  $\ell\neq p$ then $A_0(X)\xrightarrow{\cdot \ell^r} A_0(X)$ is an isomorphism for all $r \in \N$.
\end{lemma}

\begin{proof}
The first statement is classic, see \cite[Example 1.6.6]{Fulton}, the argument is as follows: since $A_0(X)$ is generated by the image of the maps of the form: 
\begin{align*}
    f_*: A_0(C) &\to A_0(X) \\
    [P]-[Q]& \mapsto f_*([P]-[Q])
\end{align*}
where $f:C \to X$ is a smooth projective curve with $P$, $Q$ points in $C$. Since $A_0(C)\simeq J(C)$ and the Jacobian of a smooth projective curve is divisible over an algebraically closed field $k$, we obtain the desired result. We prove the second statement. Note that by assuming that $k$ is an 
algebraically closed field, we get that $\CH^d(X)\simeq \CH^d_L(X)$ and that $\CH^d_L(X)\{\ell\}\simeq H^{2d-1}_\et(X,\Q_\ell/\Z_\ell(d))$. So
 \begin{align*}
      \CH_0(X)\{\ell\}=\CH^{2d}(X)\{\ell\}\simeq H^{2d-1}_\et(X,\Q_\ell/\Z_\ell(d))=0
 \end{align*}
and $\CH_0(X)\{\ell\}\simeq A_0(X)\{\ell\}$, so one deduces that under the assumption, $A_0(X)$ is $\ell^r$-divisible for any $r>0$.
\end{proof}

\begin{remark}
    Notice that with the previous statement, if $H^{2d-1}_\et(X,\Q_\ell/\Z_\ell(d))=0$ for all $\ell$ different from the characteristic of $k$, we conclude that $A_0(X)$ is uniquely $\ell^r$-divisible.
\end{remark}

For $X$ a smooth and projective variety over a field $k$ of exponential characteristic equal to $p$, we set 
\begin{align*}
    A_{0, \et}(X):= \text{ker}\left\{\deg_\et:\CH^{d}_\et(X) \to \Z[1/p] \right\}.
\end{align*}
Notice that if $k$ is algebraically closed then we have an isomorphism $ A_{0, \et}(X)\simeq A_0(X)[1/p]$.

\begin{prop}\label{propo}
Let $X$ be a geometrically integral smooth projective variety of dimension $d\geq 2$ over a perfect field $k$ with $\text{cd}(k)\leq 1$ and $p$ the exponential characteristic of $k$. Let $\bar{k}$ be the algebraic closure of $k$ and assume that  $H^{2d-1}_\et(X_{\bar{k}},\Q_\ell/\Z_\ell(d))=0$ for every prime $\ell\neq \text{char}(k)$, then $\deg_\et:\CH^d_\et(X)\to \Z[1/p]$ is surjective.
\end{prop}

\begin{proof}
First assume that $\text{char}(k)=0$, then $\CH^n_L(X)\simeq \CH^n_\et(X)$ for all $n\in \N$. Using the notation given in Lemma \ref{proHS}, if $\text{cd}(k)\leq 1$, then $E^{2,s}_2(n)=0$ for $1<s<2n$. Thus, using the characterisations of the $E_\infty$-terms of the spectral sequence given in Example \ref{remdeg}, we obtain a short exact sequence
$0\to H^1(G,H^{2n-1}_L(X_{\bar{k}},\Z(n)))\to \CH^n_L(X)\to \CH^n_L(X_{\bar{k}})^{G_k}\to 0$. For $n=d$
we have that $\CH^d_L(X)\to \CH^d_L(X_{\bar{k}})^{G_k}$ is always surjective, now consider the short exact sequence
$$0\to A_{0}(X_{\bar{k}})\to \CH^d_L(X_{\bar{k}})\xrightarrow{\deg_\et} \Z\to 0$$ where $A_{0}(X_{\bar{k}}):= \text{ker}\left\{\deg_\et:\CH^{d}_\et(X_{\bar{k}}) \to \Z\right\}$, i.e. the numerically trivial zero cycles of $X_{\bar{k}}$,  which induces a long exact sequence
\begin{align*}
    0\to A_{0}(X_{\bar{k}})^{G_k}\to \CH^d_L(X_{\bar{k}})^{G_k}\xrightarrow{\widetilde{\deg}} \Z\to H^1(G_k,A_{0}(X_{\bar{k}}))\to \ldots
\end{align*}
Where the factor $\Z$ is obtained by using the fact that  $\CH^0(\spc(\bar{k}))^{G_k}\simeq \CH^0(\spc(k))$. By \cite[Proposition 3.1(a)]{RS} we have that $\CH^d_L(X_{\bar{k}})\{\ell\}\simeq  H^{2d-1}_\et(X_{\bar{k}},\Q_\ell/\Z_\ell(d))$ so $A_{0}(X_{\bar{k}})_\text{tors}\simeq \CH^d_L(X_{\bar{k}})_\text{tors}=0$ and then the group $A_{0}(X_{\bar{k}})$ is uniquely divisible. Therefore we have that $H^1(G,A_{0}(X_{\bar{k}}))=0$ and consequently the map $\deg_\et: \CH^d_L(X)\to \CH^d_L(X_{\bar{k}})^{G_k}\to \Z$ is surjective.

Now suppose that $\text{char}(k)=p>1$, in which case it is necessary to invert the exponential characteristic of the field. For an abelian group $A$ we set $A[1/p]:= A \otimes_\Z \Z[1/p]$. Setting $s \neq 2d$ we have that $H^s_L(X_{\bar{k}},\Z(d))$ is an extension of a divisible group $D$ by a torsion group $T$. Using the convention for tensor products, we see that 
\begin{align*}
    0 \to D \to H^s_L(X_{\bar{k}},\Z(d))[1/p] \to T[1/p] \to 0
\end{align*}
where the last map kills the $p$-primary part of the torsion group $T$. The spectral sequence also holds for the complex  of \'etale sheaves $\Z[1/p](n)_\et$, for the convergence we use the same arguments with the exact triangle $\displaystyle \Z[1/p]_X(d)_\et \to \Q_X(d)_\et\to \bigoplus_{\ell \neq \text{char}(k)} \Q_\ell/\Z_\ell(d)\xrightarrow{+1}$ therefore we have a similar short exact sequence
$0\to H^1(G_k,H^{2n-1}_L(X_{\bar{k}},\Z[1/p](n)))\to \CH^n_L(X)[1/p]\to \CH^n_L(X_{\bar{k}})[1/p]^{G_k}\to 0$ and also $0\to A_{0}(X_{\bar{k}})[1/p]\to \CH^d_L(X_{\bar{k}})[1/p]\xrightarrow{\deg_L} \Z[1/p]\to 0$ therefore we can conclude.
\end{proof}

\begin{theorem}\label{teo1}
There exist a smooth projective surface $S$ over a field $k$, with $\text{char}(k)=0$ of cohomological dimension $\leq 1$, without zero cycles of degree one but $I_\et(X)=1$.   
\end{theorem}

\begin{proof}

By \cite[Th\'eor\`eme 1.1]{C-TM} and \cite[Th\'eor\`eme 1.2]{C-TM} there exist del Pezzo surfaces of degree 2, 3 and 4 over a field $k$ of characteristic zero and $\text{cd}(k)=1$ without zero cycles of degree 1. Let $S$ be one of such surfaces of degree $d\in \{2,3,4\}$. Since $S$ is a del Pezzo surface, for every field extension $K/k$ the variety $S_K$ is also a del Pezzo surface of degree $d$, in particular for $K=\bar{k}$. Since $S_{\bar{k}}$ is del Pezzo, we have that $H^1(S_{\bar{k}},\mathcal{O}_{S_{\bar{k}}})=H^2(S_{\bar{k}},\mathcal{O}_{S_{\bar{k}}})=0$ so $\text{Alb}(S_{\bar{k}})=0$. Since we are working on an algebraically closed field, $\CH^2(S_{\bar{k}})\simeq \CH^2_L(S_{\bar{k}})$ and then by Roitman's theorem which says that $\CH_0(S_{\bar{k}})_\text{tors}\simeq \text{Alb}(S_{\bar{k}})_\text{tors}$, for a proof see \cite{rojt} or \cite{BlochRo}, $\CH^2_L(S_{\bar{k}})_\text{tors}=N^2(S_{\bar{k}})_\text{tors}=0$, so the group $N^2(S_{\bar{k}})$ is uniquely divisible, and hence by Proposition \ref{propo} the map $\CH^2_L(S)\to \CH^2_L(S_{\bar{k}})^{G_k}\to \Z$ is surjective, while $\CH^2(S)\to \Z$ is not a surjective map.  
\end{proof}

\begin{theorem}\label{teo2}
    For each prime $p\geq 5$ there exist a field $F$ such that $\text{char}(k)=0$ with $\text{cd}(F)=1$ and a smooth projective hypersurface $X \subset \Pro^p_F$ with $I_\et(X)=1$ but index $I(X)=p$.
\end{theorem}

\begin{proof}
Consider $n \geq 2$, a field $k$ such that $\text{cd}(k)\leq 1$ and a hypersurface $X\subset \Pro^{n+1}_k$ which is geometrically integral. Consider the hypersurface $X_{\bar{k}}\subset \Pro^{n+1}_{\bar{k}}$. Then, by Lefschetz's theorem \cite[Theorem 7.1]{Mil80}, we have that $$H^{2n-1}_\et(X_{\bar{k}},\mu^{\otimes n}_{\ell^r})\simeq H^{2n+1}_\et(\Pro^{n+1}_{\bar{k}},\mu^{\otimes n+1}_{\ell^r})=0$$
for all $\ell \neq \text{char}(k)$, so $H^{2n-1}_\et(X_{\bar{k}},\Q_\ell/\Z_\ell(n))=0$, so by Proposition \ref{propo} the morphism $\deg_\et:\CH^n_\et(X)\to \Z$ is surjective. If we now fix a prime number $p\geq 5$, then by \cite[Theorem 1.1]{CT} there exists a field $F$ with $\text{cd}(F)=1$ and a smooth projective hypersurface $X \subset \Pro_F^p$ with index equal to $p$.
\end{proof}

\begin{remark}\label{examC}
\begin{enumerate}
    \item Assume that $k$ is a field with $\text{cd}(k)\leq 1$. Let $S$ be a smooth geometrically integral $k$ surface with $H^1(S,\mathcal{O}_S)=0$, so $\text{Alb}(S)=0$, so again by Roitman's theorem $\CH^2_L(S_{\bar{k}})$ is torsion free and then uniquely divisible, so $H^1(G,N^d(S_{\bar{k}}))=0$ and then $\CH^2_L(S)\to \Z$ is surjective. In general, if $N^d(X_{\bar{k}})$ is a divisible group, then $\CH^d_L(X)\to \Z$ is surjective.
    \item  The hypothesis of Proposition \ref{propo}, all hypersurfaces and complete intersection varieties have an étale zero cycle of degree 1 if the base field $k$ has cohomological dimension at most 1. Using \cite[Proposition 4.4]{ELW},  we get more examples of hypersurfaces with index greater than 1 but with surjective étale degree map, for example when $k=\C((t))$ and $X$ is given by $\{Y^4+tZ^4+t^2T^4+t^3W^4=0\}$.
\end{enumerate} 

\end{remark}

\subsection{\'Etale degree of Severi-Brauer varieties}

In the following, we will see non-trivial examples where the \'etale degree map is not surjective.  To do this, we will study the Lichtenbaum cohomology groups of Severi-Brauer varieties by giving an explicit characterisation of the zero cycles of the Lichtenbaum groups of Severi-Brauer varieties.

\begin{defi}
  A variety $X$ over a field $k$ is called a Severi-Brauer variety of dimension $n$ if and only if $X_{\bar{k}}\simeq \Pro^n_{\bar{k}}$. If $X$ is a Severi-Brauer variety of dimension $n$ and there exists an algebraic extension $k\subset k' \subset \bar{k}$ such that $X_{k'}\simeq \Pro^n_{k'}$ we say that $X$ splits over $k'$.
\end{defi}

If for a field $k$, such that, the Brauer group $\text{Br}(k)=0$, there exists a unique Severi-Brauer variety modulo isomorphisms to $\Pro_k^n$. Some cases of such  fields are the following:
\begin{itemize}
    \item a field $k$ with $\text{cd}(k)\leq 1$. In this category we can find fields such as algebraically and separably closed fields, finite fields, extensions of transcendence degree 1 of an algebraically closed field.
    \item if $k$ is a field extension of $\Q$ containing all the roots of unity, see \cite[\S 7]{ser} and \cite[II.\S 3, Proposition 9]{ser2}.
\end{itemize}

\begin{lemma}\label{inv}
    Let $X$ be a Severi-Brauer variety of dimension $d$ over $k$ which splits over a field $k'$. Then for all $0\leq n \leq d$ the group $\CH^n(X_{k'})\simeq \CH^n(\Pro^d_{k'})$ is a trivial $\text{Gal}(k'/k)$-module.
\end{lemma}

\begin{proof}
    Let us prove the statement by induction. The case $\CH^0(\Pro^d_{k'})$ is trivial and the case when $n=1$ is covered in the proof of \cite[Proposition 5.4.4]{GS06}. Suppose $k'/k$ is a finite Galois extension and let $G:=\text{Gal}(k'/k)$ be its associated Galois group. We have that $\text{Pic}(X_{k'})\simeq \text{Pic}(\Pro^d_{k'})\simeq \Z$, there are two possibilities for the action of $G$ on $\text{Pic}(X_{k'})$: the trivial action and the action of permuting $1$ with $-1$. The class of the line bundle in $\text{Pic}(\Pro^d_{k'})$ associated with $1$ has a global section, while the line bundle associated with $-1$ does not, so the only action you can get is the trivial one. For $n = 2$ we first consider the generator $\xi_{k'} \in \CH^1(\Pro_{k'}^d)$ and the isomorphism of Chow groups induced by the intersection with the hyperplane $\xi_{k'}$, which has the form
 $\displaystyle \CH^1(\Pro_{k'}^d) \xrightarrow[\simeq]{\cdot \xi_{k'}}\CH^2(\Pro_{k'}^d)$, so we have that
 \begin{align*}
     \CH^1(\Pro_{k}^d)\simeq \CH^1(\Pro_{k'}^d)^G\xrightarrow{(\cdot \xi_{k'})^G}\CH^2(\Pro_{k'}^d)^G
 \end{align*}
 is an isomorphism, which tells us that $G$ trivially acts on $\CH^2(\Pro_{k'}^d)$. The same kind of argument applies to $n\geq 3$. 
\end{proof}

\begin{remark}\label{remGI}
\begin{enumerate}
    \item We can similarly deduce that for all $m, n \in \N$ the group $\text{Pic}(\Pro^m_{\bar{k}}\times \Pro^n_{\bar{k}})\simeq \Z[\alpha]\oplus \Z[\beta]$, where $\alpha$ and $\beta$ are the generators of $\text{Pic}(\Pro^m_{\bar{k}})$ and $\text{Pic}(\Pro^n_{\bar{k}})$ respectively, is a trivial $G_k$-module.
    \item Let $k$ be a perfect field of characteristic exponent equals to $p$ and let $X$ be a Severi-Brauer variety of dimension $d$ over $k$. The fact $X_{\bar{k}}\simeq \Pro^d_{\bar{k}}$ simplifies several computations for the Hochschild-Serre spectral sequence given in Lemma \ref{proHS}. For instance if $m\neq 2n+1$, then for $\ell \neq p$ we can characterise the $\ell$-primary torsion groups as follows
\begin{align*}
    H^m_L(X_{\bar{k}},\Z(n))\{\ell\}\simeq H^{m-1}_\et(\Pro_{\bar{k}}^d,\Q_\ell/\Z_\ell(n)) \simeq \begin{cases}
        \Q_\ell/\Z_\ell \text{ if $m$ is odd}\\
        0 \text{ otherwise.}
    \end{cases}
\end{align*}
Therefore, for $m<2n$ and even the group $H^m_L(X_{\bar{k}},\Z(n))$ is uniquely divisible, so some of the $E_2$-terms associated with the Hochschild-Serre spectral sequence of $H^{r+s}_L(X,\Z[1/p](n))$ given in (\ref{spectralfieldf}) can be characterised as follows
\begin{align*}
    E^{r,s}_2(n) = \begin{cases}
        \quad H^s(\mathbb{P}_{\bar{k}},\Z[1/p](n))^{G_k} \quad \text{ if $r=0$,} \\
        H^r(G_k,H^{s-1}_\et(\mathbb{P}_{\bar{k}},(\Q/\Z)'(n))) \ \text{ if $s$ is odd and $r>0$,} \\
       \quad  0 \quad  \text{ if $s$ is even and $r>0$.}
    \end{cases}
\end{align*}
\end{enumerate}
\end{remark}

Now let $n=1$ and $X$ be a Severi-Brauer variety over $k$ of dimension $d$. Using the Hoschschild-Serre spectral sequence given in Lemma \ref{proHS} and Lemma \ref{inv}, we recover a classical result of Lichtenbaum, see \cite[Theorem 5.4.10]{GS06}, concerning the Picard group of $X$ and Brauer groups there is an exact sequence
\begin{align}\label{lich}
    0 \to \text{Pic}(X) \to  \text{Pic}(\Pro_{\bar{k}}^d)^{G_k} \xrightarrow{\delta} \text{Br}(k) \to \text{Br}(k(X)),
\end{align}
where the map $\delta$ sends 1 to the class of $X$ in $\text{Br}(k)$. For an arbitrary integer $n$, using the projective bundle formula (\ref{proBun}), we get
\begin{align*}
    H^m_L(\Pro_{\bar{k}}^d,\Z(n)) \simeq \bigoplus^{d}_{i=0} H^{m-2i}_L(\spc(\bar{k}),\Z(n-i)).
\end{align*}
As we change the base field to its algebraic closure, we have that $H^m_L(\Pro_{\bar{k}}^d,\Z(d)) \simeq H^m_M(\Pro_{\bar{k}}^d, \Z(d)) $ for all $m \in \Z$ and in particular $H^{m-2i}_L(\spc(\bar{k}),\Z(d-i)) =0$ if $m-d>i$. For example, if $m=2d-1$ then $H^{2d-1}_L(\Pro_{\bar{k}}^d,\Z(d))\simeq K_1^M(\bar{k})$ or for $m=2d-2$ we have $H^{2d-2}_L(\Pro_{\bar{k}}^d, \Z(d))\simeq K_2^M(\bar{k})$ and so for a Severi-Brauer variety and applying Lemma \ref{proHS} we get that $E^{1,2d-1}_2(d)=H^1(G_k,\bar{k}^*)=0$, by Hilbert 90, 
and $E^{2,2d-1}_2(d)=H^2(G_k,\bar{k}^*)=\text{Br}(k)$.

\begin{example}\label{exam}
Let $X$ be a Severi-Brauer variety of dimension $d=2$ over a perfect field $k$ with Galois group $G_k$. Using the previous characterisations through the projective bundle formula, we then describe the $E_2$-terms associated to $X$ in the following way:
    \begin{align*}
        & E^{r,0}_2(2)=H^r(G_k,H^0_M(\spc(\bar{k}),\Z(2))), \  E^{r,1}_2(2)=H^r(G_k,H^1_M(\spc(\bar{k}),\Z(2))), \\ 
        & E^{r,2}_2(2)=H^r(G_k,K_2^M(\bar{k})), \ E^{r,3}_2(2)=H^r(G_k,K_1^M(\bar{k})), \\ 
        & E^{r,4}_2(2)=H^r(G_k,\CH^2_L(\mathbb{P}^2_{\bar{k}})) \text{ and } E^{r,s}_2(2) = 0 \text{ for }s\geq 5.
    \end{align*}
By Remark \ref{remGI} (2), we have that $E^{r,0}_2(2)= E^{r,2}_2(2)=0$ for $r>0$, also $E^{1,3}_2(2)=0$ by Hilbert 90 theorem and $E^{2,3}_2(2)=\text{Br}(k)$, obtaining with this the following terms: by trivial reasons $E^{1,3}_\infty(2) = E^{2,2}_\infty(2) = E^{4,0}_\infty(2) =0$ and:
\begin{align*}
        E^{3,1}_\infty(2)=H^3(G_k,H^1_M(\bar{k},\Z(2)))/ \text{im}\left\{K_1^M(\bar{k})^{G_k}\to H^3(G_k,H^1_M(\bar{k},\Z(2)))\right\}
\end{align*}
The only remaining piece of the filtration of $\CH_L^2(X)$  that we need to study is $E^{0,4}_\infty(2)$.  By definition we have that
$E^{0,4}_3(2)=\text{ker}\left\{\CH^2(\Pro^2_{\bar{k}})^{G_k} \to \text{Br}(k)\right\}$ and as $E^{3,2}_2(2)=0$ then $ E^{0,4}_4(2)=E^{0,4}_3(2)$. Finally, we observe that 
 $E^{4,1}_4(2)=E^{4,1}_3(2)=E^{4,1}_2(2)$ and thus again by definition 
 \begin{align*}
     E^{0,4}_\infty(2) &= \text{ker}\left\{E^{0,4}_4(2) \to E^{4,1}_4(2)\right\} \\
     &=\text{ker}\left\{E^{0,4}_4(2) \to H^4(G_k,H^1_M(\spc(\bar{k}),\Z(2)))\right\}.
 \end{align*}
Therefore $\CH^2_L(X)$ fits into a short exact sequence given by the filtration induced by the 
Hochschild-Serre spectral sequence
\begin{align*}
    0 \to E^{3,1}_\infty(2) \to \CH^2_L(X) \to E^{0,4}_\infty(2) \to 0.
\end{align*}
\end{example}

If we want to generalise the previous example for higher dimension, we need to impose a condition on the cohomological dimension of $k$, as stated in the following theorem:

\begin{theorem}\label{teoSV}
    Let $X$ be a Severi-Brauer variety of dimension $d$ over a field $k$. Then the image of $\deg_\et:CH^d_\et(X)\to \Z$ is isomorphic to a subgroup of $\text{Pic}(X)$ and in particular $I_\et(X)\geq \text{ord}([X])$ where $[X]$ is the Brauer class of $X$ in $\text{Br}(k)$. Moreover, if $\text{cd}(k)\leq 4$, then the group $\CH_0^L(X)$  fits in the following exact sequence
    \begin{align*}
         0 \to E^{3,2d-1}_\infty(d) \to \CH^d_L(X) \to E^{0,2d}_\infty(d) \to 0.
    \end{align*}
with the $E_2$-terms come from the spectral sequence (\ref{spectralfield}), where $E^{0,2d}_\infty(d)= \text{ker}\left\{\CH^d(\mathbb{P}^d_{\bar{k}})^{G_k}\to \text{Br}(k)\right\}$. In particular one gets $I_\et(X)=\text{ord}([X])$.
\end{theorem}

\begin{proof}
Before we start, let us mention that for simplicity we will use the Hochschild-Serre spectral sequence (\ref{spectralfield}) to get the expressions for the Lichtnebaum cohomology groups, and then we will invert the characteristic $p$ to use étale motivic cohomology.

If we consider a smooth projective variety $X$ over $k$ and we fix an integer $1 \leq n \leq d$ from the projective bundle formula (\ref{proBun}) for Lichtenbaum cohomology we have that the isomorphism
\begin{align*}
    H^{2n-1}_L(\Pro_{\bar{k}}^d,\Z(n))\simeq \bigoplus_{i=0}^d H^{2(n-i)-1}_L(\bar{k},\Z(n-i))
\end{align*}
so in particular we obtain $H^{2n-1}_L(\Pro^d_{\bar{k}},\Z(n))\simeq K_1^M(\bar{k})$. Together with the Hochschild-Serre spectral sequence given in (\ref{spectralfield}) associated to $\CH^n_L(X)$ one gets $E^{0,2n}_2(n)\simeq \text{Br}(k)$ and consequently by looking at the following commutative diagram
\begin{equation}\label{commutative}
     \begin{tikzcd}
 \text{Pic}(\Pro_{\bar{k}}^d)^{G_k} \arrow{r}{\delta} \arrow{d}{\simeq} & \text{Br}(k) \arrow{d}{\simeq} \\
   \CH^n(\Pro_{\bar{k}}^d)^{G_k} \arrow{r}{d_2^{0,2n}(n)} & \text{Br}(k)
  \end{tikzcd}   
\end{equation}
we conclude that the term $E^{0,2n}_\infty(n)$ is isomorphic to a subgroup of $\text{Pic}(X)$.

Now let us set $n=d$, since the vertical arrows in the commutative diagram (\ref{commutative}) are isomorphisms, then $E^{0,2d}_3(d)= \text{ker}(d_2^{0,2d}(d))\simeq \text{ker}(\delta) \simeq \text{Pic}(X)$. Now by Proposition \ref{rem}, the map $\deg_\et$ factors through $E_\infty^{0,2d}(d)$ which is a subgroup of $E^{0,2d}_3(d)$, thus $I_\et(X)\geq \text{ord}([X])$ by the long exact sequence (\ref{lich}).

If we impose that $\text{cd}(k)\leq 4$, to get an expression for $\CH_0^L(X)$, we follow the arguments given in Example \ref{exam}.    Consider such $k$ and $X$, then by hypothesis $X_{\bar{k}}\simeq \mathbb{P}^d_{\bar{k}}$. Again by the projective bundle formula for Lichtenbaum cohomology we have that 
\begin{align*}
    H^m_L(\mathbb{P}^d_{\bar{k}},\Z(d))\simeq \bigoplus_{i=0}^d H^{m-2i}_L(\spc(\bar{k}),\Z(d-i)).
\end{align*}

Notice that by divisibility arguments we have 
 that $E^{p,2k}_2(d)=0$ for $0\leq k\leq d$ and $p>0$. Under the assumption about the cohomological dimension of $k$ we have that $E^{p,q}_2(d)=0$ for $p>4$ and $q<2n$, this results that $E^{0,2d}_\infty(d) \simeq E^{0,2d}_3(d)=\text{ker}\left\{\CH^d(\mathbb{P}^d_{\bar{k}})^{G_k}\to H^2(G_k,\bar{k}^*)\right\}$ and the other terms $E^{p,q}_2(d)$ with $p+q=2d$ that could not vanish are $E^{1,2d-1}_2(d)$ and $E^{3,2d-3}_2(d)$, but 
 $H^{2d-1}_L(\Pro_{\bar{k}}^d,\Z(d))\simeq K_1^M(\bar{k})$ therefore $E^{1,2d-1}_2(d)=0$. On the other hand, the remaining element of the filtration, which is $E^{3,2d-3}_\infty(d)= E^{3,2d-3}_4(d)$, is defined as
\begin{align*}
    E^{3,2d-3}_4(d)&=E^{3,2d-3}_3(d)/\text{im}\left\{E^{0,2d-1}_3(d) \to E^{3,2d-3}_3(d)\right\} \\
    &=H^3(G_k,H^{2d-3}_M(\Pro_{\bar{k}}^d,\Z(d)))/ \text{im}\left\{K_1^M(\bar{k})^{G_k}\to H^3(G_k,H^{2d-3}_M(\Pro_{\bar{k}}^d,\Z(d)))\right\}.
\end{align*}
Using the recursive formula $$ H^m_L(\mathbb{P}^n_{\bar{k}},\Z(n))\simeq H^m_L({\bar{k}},\Z(n))\oplus H^{m-2}_L(\mathbb{P}^{n-1}_{\bar{k}},\Z(n-1)).$$ 
 we can obtain easily 
\begin{align*}
    H^{2d-3}_M(\Pro_{\bar{k}}^d,\Z(d)) \simeq \begin{cases}
        0 \text{ if }d=1 \\
        H^1_M(\bar{k},\Z(2)) \text{ if }d=2 \\
         H^1_M(\bar{k},\Z(2)) \oplus K_3^M(\bar{k}) \text{ if }d\geq 3. 
    \end{cases}
\end{align*}

Again as in Example \ref{exam}, the group $\CH^d_L(X)$ fits into the following short exact sequence
\begin{align}\label{chd}
    0 \to  E^{3,2d-3}_\infty(d) \to \CH^d_L(X) \to E^{0,2d}_\infty(d) \to 0. 
\end{align}

We should emphasise that as we mentioned in Proposition \ref{rem}, the \'etale degree map factors through $E^{0,2d}_\infty(d)[1/p]$, having the following commutative diagram, which is obtain if we tensor the short exact sequence (\ref{chd}) by $\Z[1/p]$
\[  
    \begin{tikzcd}
0 \arrow{r} & E^{3,2d-3}_\infty(d)[1/p] \arrow{r} & \CH^d_\et(X) \arrow[swap]{dr}{\deg_\et}\arrow{r} &  E^{0,2d}_\infty(d)[1/p] \arrow{d}{\widetilde{\deg}}\arrow{r} & 0 \\
& & &  \Z[1/p], & 
  \end{tikzcd}
\]
and where $\widetilde{\deg}:E^{0,2d}_\infty(d)[1/p] \to \Z[1/p]$ is the composition of the following maps:
\begin{align*}
    E^{0,2d}_\infty(d)[1/p] \hookrightarrow \CH^d(\Pro^d_{\bar{k}})[1/p]^{G_k} \xrightarrow{\simeq} \CH^d(\Pro^d_{\bar{k}})[1/p] \xrightarrow{\deg} \Z[1/p].
\end{align*}

The assumption about the cohomological dimension of $k$ gives us that $E^{0,2d}_\infty(d)[1/p]\simeq E^{0,2d}_3(d)[1/p]$, thus we conclude that $I_\et(X)=\text{ord}([X])$.
\end{proof}

As we may expect, the \'etale index of a  product of Severi-Brauer is again bounded by the order of the Brauer class in $\text{Br}(k)$. For the sequel we denote $X^{\times n}:=\overbrace{X\times \ldots \times X}^{\text{n-times}}$

\begin{lemma}\label{lemBr}
    Let $X$ be a Severi-Brauer variety of dimension $d$ over a field $k$. Then there exists an exact sequence
\begin{align*}
    0 \to \text{Pic}(X\times X) \to \text{Pic}(\Pro^d_{\bar{k}} \times \Pro^d_{\bar{k}})^{G_k} \simeq  \Z \oplus \Z \xrightarrow{s} \text{Br}(k) \to \text{Br}(X \times X)
\end{align*}
where $s$ sends $(a,b) \mapsto (a+b)\left[X\right] \in \text{Br}(k)$ and $\left[X\right]$ is the Brauer class associated to $X$. In general for a product $X^{\times n}$ we then obtain an exact sequence
 \begin{align*}
    0 \to \text{Pic}(X^{\times n}) \to \text{Pic}(\Pro^d_{\bar{k}} \times \ldots \times \Pro^d_{\bar{k}})^{G_k} \simeq  \Z\oplus \ldots \oplus \Z \xrightarrow{s} \text{Br}(k) \to \text{Br}(X^{\times n})
\end{align*}
where $s$ sends $(a_1,\ldots,a_n) \mapsto \sum_{i=1}^n a_i\left[X\right] \in \text{Br}(k)$.
\end{lemma}

\begin{proof}
Let $Y$ be a smooth projective variety over $k$. Considering the Hochschild-Serre spectral sequence (\ref{spectralfield})
\begin{align*}
    E^{r,s}_2(1) = H^r(G_k,H^s_L(Y_{\bar{k}},\Z(1))) \Longrightarrow H^{r+s}_L(Y,\Z(1))
\end{align*}
we obtain the following exact sequence $0 \to E^2_\infty(1) \to E^{0,2}_2(1) \to E^{2,1}_2(1) \to E^{3}_\infty(1)$. If $Y=X^{\times n}$ then $Y_{\bar{k}}\simeq \Pro^d_{\bar{k}} \times \ldots \times \Pro^d_{\bar{k}}$ and consequently $\text{Pic}(\Pro^d_{\bar{k}} \times \ldots \times \Pro^d_{\bar{k}})\simeq \Z \oplus \ldots \oplus \Z$ and by Remark \ref{remGI} we obtain an isomorphism $\text{Pic}(\Pro^d_{\bar{k}} \times \ldots \times \Pro^d_{\bar{k}})^{G_k} \simeq \Z \oplus \ldots \oplus \Z$ giving us the exact sequences of the statement. 

Now let us see the easiest case for $Y=X\times X$. Consider the arrows
\[ 
\begin{tikzcd}
X \arrow[r,"\Delta"]
& X \times X \arrow[r, shift left,"\text{pr}_1"]
\arrow[r, shift right, "\text{pr}_2", swap]
& X 
\end{tikzcd}
\]
where $\Delta: X \to X\times X$ is the diagonal embedding and $\text{pr}_i:X\times X \to X $ is the projection in the i-th component. Notice that the composition gives the identity on $X$. Notice that the morphism $\text{pr}_i:X\times X \to X $ induces a morphism 
\begin{align*}
    \text{pr}_i^*:H^m_L(X,\Z(n)) \to H^m_L(X\times X,\Z(n)) \hspace{2mm}\text{and}\hspace{2mm} \text{pr}_i^*:H^m_L(\Pro^d_{\bar{k}},\Z(n)) \to H^m_L(\Pro^d_{\bar{k}}\times \Pro^d_{\bar{k}},\Z(n))
\end{align*}
for all bi-degree $(m,n)$. By functoriality properties of the Hochschild-Serre spectral sequence we have a diagram
\[ 
\begin{tikzcd}
    0 \arrow{r}& \text{Pic}(X) \arrow{r} \arrow{d}{} & \Z \arrow{r} \arrow{d}{f} & \text{Br}(k) \arrow{r}\arrow{d}{\tilde{f}} & \text{Br}(X) \arrow{d}{}\\
    0 \arrow{r}& \text{Pic}(X\times X) \arrow{r} & \Z \oplus \Z \arrow{r}{s} & \text{Br}(k) \arrow{r}& \text{Br}(X \times X)
\end{tikzcd}
\]
where the vertical arrows are induced by $ \text{pr}_i^*$. The composition $\text{pr}_i\circ \Delta$ is the identity on $X$, thus $\text{id}^*=\Delta^*\circ \text{pr}_i^*$ therefore we obtain that the maps $f:\Z \to \Z \oplus \Z$ and  $\tilde{f}:\text{Br}(k) \to \text{Br}(k)$ are injective and then, the elements of the form $(a,0)$ and $(0,b)$ are sent through $s$ to $a\left[X \right]$ and $b\left[X \right] \in \text{Br}(k)$ respectively. For the general case, we use the arrows
\[ 
\begin{tikzcd}
X \arrow[r,"\tilde{\Delta}"]
& \overbrace{X  \times \ldots \times X}^{\text{n-times}} \arrow[r, shift left=3,"\text{pr}_1"]
\arrow[r, draw=none, "\vdots",shift right=2.5]
\arrow[r, shift right=3, "\text{pr}_n", swap]
& X 
\end{tikzcd}
\]
 where $\tilde{\Delta}$ is the n-diagonal morphism and $\text{pr}_i$ is the projection in the $i$-th component, then we conclude as in the case of $X\times X$.
\end{proof}

\begin{theorem}\label{teo}
    Let $k$ be a field and let $X$ be a Severi-Brauer variety over $k$ of dimension $d$. Then $I_\et(X^{\times n}) \geq I_\et(X) \geq \text{ord}([X])$.
\end{theorem}

\begin{proof}
Let us fix an integer $n\geq 1$, and let $X$ be a Severi-Brauer variety of dimension $d$ over $k$. We have that $X_{\bar{k}}\simeq \Pro_{\bar{k}}^d$, thus $(X^{\times n})_{\bar{k}}\simeq \Pro_{\bar{k}}^d \times \ldots \times \Pro_{\bar{k}}^d$. Considering the Hochschild-Serre spectral sequence (\ref{spectralfield}) for the Lichtenbaum cohomology groups of $X^{\times n}$, one has
    \begin{align*}
    E^{r,s}_2(nd) = H^r(G_k,H^s_L((\mathbb{P}^d_{\bar{k}})^{\times n},\Z(nd))) \Longrightarrow H^{r+s}_L(X^{\times n},\Z(nd)).
    \end{align*}
Notice that by the projective bundle formula (\ref{proBun}) for Lichtenbaum cohomology we have   
\begin{align*}
    H^{2nd-1}_L\left((\mathbb{P}^d_{\bar{k}})^{\times n},\Z(nd)\right)\simeq \bigoplus_{0\leq a_1,\ldots a_n\leq d }^n  H^{2nd-1-2\sum_{j=1}^n a_j}_L\left(\spc(\bar{k}),\Z\left(nd-\sum_{j=1}^n a_j\right)\right).
\end{align*}
By noticing that if $2nd-1-2\sum_{j=1}^d a_j>nd-\sum_{j=1}^n a_j$ then $H^{2nd-1-2\sum_{j=1}^n a_j}_L(\spc(\bar{k}),\Z(nd-\sum_{j=1}^n a_j))=0$, this give us a vanishing condition  for $nd-1>\sum_{j=1}^n a_j$. As $0 \leq a_j \leq d$ for all $j$, then the only n-tuples $(a_1,\ldots, a_n)$ which do not satisfy such condition are 
\begin{align*}
    \epsilon_i=(d,\ldots, d,\overbrace{d-1}^{i\text{-th pos.}},d,\ldots, d) \text{ for all $i$, and } (d,\ldots,d).
\end{align*}
 For such cases, if $a_j=d$ for all $j$ then  $H^{2nd-1-2nd}_L(\spc(\bar{k}),\Z(nd-nd))=H^{-1}_L(\spc(\bar{k}),\Z(0))=0$,
and if $(a_1,\ldots,a_n)=\epsilon_i$, then 
\begin{align*}
    H^{2nd-1-2\sum_{j=1}^n a_j}_L(\spc(\bar{k}),\Z(nd-\sum_{j=1}^n a_j))=H^1(\spc(\bar{k}),\Z(1))\simeq K^M_1(\bar{k})=\bar{k}^*
\end{align*}
thus $ H^{2nd-1}_L((\mathbb{P}^d_{\bar{k}})^{\times n},\Z(nd))\simeq \bigoplus_{i=1}^n \bar{k}^*$ and consequently $E^{2,2nd-1}_2(nd)\simeq \bigoplus_{i=1}^n  \text{Br}(k)$. We have that the term $E^{0,2nd}_3(nd)$ is isomorphic to $\text{ker}\left\{\CH^{nd}((\mathbb{P}^d_{\bar{k}})^{\times n})^{G_k} \xrightarrow{g}  \bigoplus_{i=1}^n  \text{Br}(k)\right\},$ and consider the element $\delta$ defined as $$\delta=c_1\left(\mathcal{O}_{\mathbb{P}^d_{\bar{k}}\times \ldots \times \mathbb{P}^d_{\bar{k}}}(1)\right)^{nd-1}=\sum_{\substack{a_1,\ldots,a_n\in \{d-1,d\}  \\ a_1+\ldots+ a_n=nd-1}} x_1^{a_1}\cdots x_n^{a_n}$$ 
and $x_i$ is the pull-back of the generator class of $\text{Pic}(\Pro^d_{\bar{k}})$ through the map $\text{pr}_i:X^{\times n}\to X$. The intersection product with $\delta$ defines morphisms
\begin{align*}
    \text{Pic}((\mathbb{P}^d_{\bar{k}})^{\times n}) \xrightarrow{ \cup \delta} \CH^{nd}((\mathbb{P}^d_{\bar{k}})^{\times n}) \hspace{2 mm} \text{ and} \hspace{2 mm} H^1_L((\mathbb{P}^d_{\bar{k}})^{\times n}),\Z(1)) \xrightarrow{ \cup \delta}H^{2nd-1}_L((\mathbb{P}^d_{\bar{k}})^{\times n}),\Z(nd)). 
\end{align*}
By the functoriality of the Hochschild-Serre spectral sequence we obtain a commutative diagram as follows:
\[ 
\begin{tikzcd}
     \text{Pic}((\mathbb{P}^d_{\bar{k}})^{\times n})^{G_k} \arrow{r}{s} \arrow{d} & \text{Br}(k)  \arrow{d}\\
    \CH^{nd}((\mathbb{P}^d_{\bar{k}})^{\times n})^{G_k} \arrow{r}{g} & \bigoplus_{i=1}^n  \text{Br}(k),
\end{tikzcd}
\]
where the vertical arrows are induced by $\delta$. According to Lemma \ref{lemBr} the arrow $s$ sends $\Z\oplus \ldots \oplus \Z \ni (\alpha_1,\ldots,\alpha_n) \mapsto \sum_{i=1}^d \alpha_i [X] \in \text{Br}(k)$. Notice that the arrow $x_i \mapsto x_1^d\cdots x_{i}^{d-1}\cdots x_n^d$ induces an isomorphism $\CH^1((\mathbb{P}^d_{\bar{k}})^{\times n})\simeq \CH^{nd-1}((\mathbb{P}^d_{\bar{k}})^{\times n})$  and that $\CH^{nd-1}((\mathbb{P}^d_{\bar{k}})^{\times n})\otimes H^1((\mathbb{P}^d_{\bar{k}})^{\times n},\Z(1)) \simeq  H^{2nd-1}_L((\mathbb{P}^d_{\bar{k}})^{\times n},\Z(nd))$ given by the map $(\alpha_1,\ldots,\alpha_n) \otimes \beta  \mapsto \beta(\alpha_1,\ldots,\alpha_n)$ which is the cup product. 

Therefore the arrow $g$ maps $ \CH^{nd}((\mathbb{P}^d_{\bar{k}})^{\times n})^{G_k} \ni a \mapsto (a[X],\ldots, a[X]) \in \text{Br}(k)$ giving us that $\text{ker}(g)=\text{ord}([X])\Z$. Since $E_\infty^{0,2nd}(nd) \hookrightarrow E_3^{0,2nd}(nd)=\text{ker}(g)$ and $\deg_\et$ factors through $E_\infty^{0,2nd}(nd)$ we conclude the proof.
\end{proof}

The natural question that arises is when this bound is reached; this is the case for the product $C\times C$ when $C$ is a smooth, geometrically connected curve of genus $0$ over a field $k$ such that $C_{\bar{k}}\simeq \mathbb{P}^1_{\bar{k}}$ as the following proposition shows:

\begin{prop}
    Let $k$ be a perfect field of characteristic $p\geq 0$ with Galois group $G_k$, and let $C$ be a smooth, geometrically connected curve of genus $0$ over the field $k$ such that $C_{\bar{k}}\simeq \mathbb{P}^1_{\bar{k}}$, then $I_\et(C\times C)= \text{ord}([C])$. 
\end{prop}

\begin{proof}
    By our assumptions we have that $C_{\bar{k}}\simeq \mathbb{P}^1_{\bar{k}}$ then $(C\times C)_{\bar{k}} \simeq \mathbb{P}^1_{\bar{k}} \times \mathbb{P}^1_{\bar{k}}$. Considering the Hochschild-Serre spectral sequence (\ref{spectralfield}) for the groups $H^m_L(C\times C,\Z(2))$ one gets 
    \begin{align*}
    E^{r,s}_2(2) = H^r(G_k,H^s_L(\mathbb{P}^1_{\bar{k}}\times \mathbb{P}^1_{\bar{k}},\Z(2))) \Longrightarrow H^{p+q}_L(C\times C,\Z(2)).
    \end{align*}
    Since $H^m_L(\mathbb{P}^1_{\bar{k}}\times \mathbb{P}^1_{\bar{k}},\Z(2))\simeq H^m_M(\mathbb{P}^1_{\bar{k}}\times \mathbb{P}^1_{\bar{k}},\Z(2))$ for $m\leq 3$, using again the projective bundle formula for motivic cohomology we obtain that
    \begin{align*}
        H^3_L(\mathbb{P}^1_{\bar{k}}\times \mathbb{P}^1_{\bar{k}},\Z(2))&\simeq H^3_M(\mathbb{P}^1_{\bar{k}},\Z(2))\oplus H^1_M(\mathbb{P}^1_{\bar{k}},\Z(1)) \simeq  K_1(\bar{k})\oplus K_1(\bar{k}) \\
         H^2_L(\mathbb{P}^1_{\bar{k}}\times \mathbb{P}^1_{\bar{k}},\Z(2))&\simeq H^2_M(\mathbb{P}^1_{\bar{k}},\Z(2))\oplus H^0_M(\mathbb{P}^1_{\bar{k}},\Z(1))\simeq  K_2(\bar{k}) \\
         H^1_L(\mathbb{P}^1_{\bar{k}}\times \mathbb{P}^1_{\bar{k}},\Z(2))&\simeq H^1_M(\mathbb{P}^1_{\bar{k}},\Z(2))\simeq  H^1_M(\spc(\bar{k}),\Z(2)) \\
        H^0_L(\mathbb{P}^1_{\bar{k}}\times \mathbb{P}^1_{\bar{k}},\Z(2))&\simeq H^0_M(\mathbb{P}^1_{\bar{k}},\Z(2))\simeq  H^0_M(\spc(\bar{k}),\Z(2)).
    \end{align*}
As we have mentioned before, $H^0_M(\spc(\bar{k}),\Z(2))$ and $K_2(\bar{k})$ are uniquely divisible, then for $r>0$ we have $E^{r,0}_2(2)=E^{r,2}_2(2)=0$. Due to the compatibility of \'etale cohomology with colimits, and in particular with direct sums, so $E^{r,3}_2(2) \simeq H^r(G_k,\bar{k}^*)\oplus H^r(G_k,\bar{k}^*)$. In particular, notice that again Hilbert's theorem 90 gives us that $E^{1,3}_2(2)=0$ and by definition $E^{2,3}_2(2)\simeq \text{Br}(k)\oplus \text{Br}(k)$.

With all this information about the $E_2$-terms, we obtain the $E_\infty$-terms that are trivial $ E^{1,3}_\infty(2) = E^{2,2}_\infty(2)= E^{4,0}_\infty(2) =0$, and those that can be non-trivial
\begin{align*}
    E^{0,4}_\infty(2) &=\text{ker}\left\{\CH^2(\mathbb{P}^1_{\bar{k}}\times \mathbb{P}^1_{\bar{k}})^{G_k} \to \text{Br}(k)\oplus\text{Br}(k)\right\}, \\
    E^{3,1}_\infty(2) &= E^{3,1}_2(2)/\text{im}\left\{E^{0,3}_2(2) \to E^{3,1}_2(2)\right\}.
\end{align*}
We then obtain the following short exact sequence 
\[  
    \begin{tikzcd}
0 \arrow{r} & E^{3,1}_\infty(2)[1/p] \arrow{r} & \CH^2_\et(C \times C) \arrow[swap]{dr}{\deg_\et}\arrow{r} &  E^{0,4}_\infty(2)[1/p] \arrow{d}{\widetilde{\deg}}\arrow{r} & 0 \\
& & &  \Z[1/p] & 
  \end{tikzcd}
\]
where $\widetilde{\deg}:E^{0,4}_\infty(2)[1/p]\to \Z[1/p]$ is the composition of the following maps:
\begin{align*}
    E^{0,4}_\infty(2)[1/p] \hookrightarrow \CH^2(\Pro^1_{\bar{k}}\times \Pro^1_{\bar{k}})[1/p]^{G_k} \xrightarrow{\simeq} \CH^2(\Pro^1_{\bar{k}}\times \Pro^1_{\bar{k}})[1/p] \xrightarrow{\deg} \Z[1/p].
\end{align*}
Let us give more information about the term $E^{0,4}_\infty(2)$. Mimicking the proof of Theorem \ref{teo}, we have an isomorphism $\text{Pic}(\Pro^1_{\bar{k}}\times \Pro^1_{\bar{k}}) \simeq \Z[x]\oplus \Z[y]$ and let us consider the Chern class $\delta =c_1(\mathcal{O}_{\Pro^1_{\bar{k}}\times \Pro^1_{\bar{k}}}(1))=x+y$. Taking  the morphisms induced by the intersection product with $\delta$:
\begin{align*}
H^1_L(\Pro_{\bar{k}}^1\times \Pro_{\bar{k}}^1,\Z(1)) \xrightarrow{\cup \delta} H^3_L(\Pro_{\bar{k}}^1\times \Pro_{\bar{k}}^1,\Z(2)) \ \text{ and } \ \CH^1(\Pro_{\bar{k}}^1\times \Pro_{\bar{k}}^1) \xrightarrow{\cup \delta}   \CH^2(\Pro_{\bar{k}}^1\times \Pro_{\bar{k}}^1). 
\end{align*}
Consider the isomorphism
\begin{align*}
    \CH^1(\Pro_{\bar{k}}^1\times \Pro_{\bar{k}}^1) \otimes H^1_L(\Pro_{\bar{k}}^1\times \Pro_{\bar{k}}^1,\Z(1)) &\xrightarrow{\simeq} H^3_L(\Pro_{\bar{k}}^1\times \Pro_{\bar{k}}^1,\Z(2)) \\
    (a,b) \otimes \alpha  &\mapsto \alpha(a,b)
\end{align*}
induced by the cup product. Hence the cup product with the diagonal induces a map $\text{Br}(k) \to \text{Br}(k)\oplus \text{Br}(k)$ defined by $a \mapsto (a,a)$ and then we can deduce that $\CH^2(\Pro_{\bar{k}}^1\times \Pro_{\bar{k}}^1)^{G_k} \to \text{Br}(k)\oplus \text{Br}(k)$ sends the $1 \mapsto ([C],[C])$. Since $E^{0,4}_\infty(2)\simeq \text{ord}([C])\Z[1/p]$ we conclude that $I_\et(C\times C)= \text{ord}([C])$. 
\end{proof}

\begin{remark}
\begin{enumerate}
    \item By Theorem \ref{teo1} and Theorem \ref{teo2} with the \'etale degree map we can improve the existence of integral projectors of the Künneth decomposition. Even though by Theorem \ref{teo} there exists $X$ such that $I_\et(X)\neq 1$ we have that $I(X)\geq I_\et(X)$. If there exists an element $e \in \CH^d_\et(X)$ of \'etale degree 1, then we define 
    \begin{align*}
        p_0^\et(X)=\text{pr}_{1}^*(e)\cdot \text{pr}_{2}^*(X) \ \text{ and } \ p_{2d}^\et(X)=\text{pr}_{1}^*(X)\cdot \text{pr}_{2}^*(e)
    \end{align*}
where $\text{pr}_i:X \times X \to X$ is the projection to the i-th factor.
    
    \item  If $k$ is a field with $\text{Br}(k)=0$, then the Severi-Brauer varieties $X$ over $k$ split over it and then $I(X)=I_\et(X)=1$. Thus, as Theorem \ref{teo} shows, $\text{Br}(k)$ appears to be an obstruction to the existence of an \'etale zero cycle of degree 1.  
\end{enumerate}

\end{remark}

\printbibliography[title={Bibliography}]

\info
\end{document}